\documentclass{amsart}

\usepackage[all]{xy}

\usepackage[a4paper, top=2cm, bottom=2cm, left=3cm, right=3cm, heightrounded, marginparwidth=3.5cm, marginparsep=.2cm, centering]{geometry}
\usepackage{hyperref}
\usepackage{cleveref}

\usepackage{tikz-cd}
\usepackage{bbold}
\usepackage{amssymb}
\usepackage{amsaddr}
\usepackage{todonotes}
\usepackage[T1]{fontenc}

\usepackage{relsize}
\usepackage[bbgreekl]{mathbbol}
\usepackage{amsfonts}

\DeclareSymbolFontAlphabet{\mathbb}{AMSb}
\DeclareSymbolFontAlphabet{\mathbbl}{bbold}
\newcommand{\prism}{{\mathlarger{\mathbbl{\Delta}}}}

\theoremstyle{plain}
\newtheorem{theorem}{Theorem}[section]
\newtheorem{proposition}[theorem]{Proposition}
\newtheorem{lemma}[theorem]{Lemma}
\newtheorem{corollary}[theorem]{Corollary}

\newtheorem{question}[theorem]{Question}
\theoremstyle{definition}
\newtheorem{definition}[theorem]{Definition}
\newtheorem{example}[theorem]{Example}
\newtheorem{remark}[theorem]{Remark}

\newcommand{\calO}{\mathcal{O}}

\newcommand{\cts}{\mathrm{cts}}
\newcommand{\cycl}{\mathrm{cycl}}

\newcommand{\en}{\mathrm{en}}
\newcommand{\et}{\mathrm{{\acute{e}t}}}

\newcommand{\Gal}{\mathrm{Gal}}
\newcommand{\Gm}{{\mathbb{G}_m}}
\newcommand{\Ga}{\mathbb{G}_a}

\newcommand{\Hom}{\mathrm{Hom}}
\newcommand{\HT}{\mathrm{HT}}

\newcommand{\Id}{\textrm{Id}}

\renewcommand{\inf}{\mathrm{inf}}

\newcommand{\N}{\mathbb{N}}
\newcommand{\Nilp}{\mathrm{Nilp}}

\newcommand{\calPerf}{\mathcal{P}\!\textit{erf}\,}

\newcommand{\Q}{\mathbb{Q}}
\newcommand{\Rep}{\mathrm{Rep}}

\newcommand{\Sen}{\mathrm{Sen}}

\newcommand{\Spa}{\mathrm{Spa}}

\newcommand{\Spf}{\mathrm{Spf}}
\newcommand{\Spec}{\mathrm{Spec}}

\newcommand{\alst}{\alpha}
\newcommand{\alstp}{\alpha^+}
\newcommand{\WCart}{\mathrm{WCart}}
\newcommand{\Z}{\mathbb{Z}}

\newcommand{\Br}{\mathrm{Br}}
\newcommand{\Pic}{\mathrm{Pic}}
\renewcommand{\O}{\mathcal{O}}
\renewcommand{\Vec}{\mathrm{Vec}}
\newcommand{\tf}{[\tfrac{1}{p}]}

\newcommand{\xisomarrow}[1]{%
	\xrightarrow[\raisebox{0.3em}{\smash{\ensuremath{\sim}}}]{#1}
}

\begin{document}

	\title{\MakeLowercase{v}-vector bundles on $p$-adic fields and Sen theory via the Hodge--Tate stack}
	
	\author[J. Ansch\"utz, B. Heuer, A.-C. Le Bras]{Johannes Ansch\"utz, Ben Heuer, Arthur-C\'esar Le Bras}

	\begin{abstract}
        We describe the category of continuous semilinear representations and their cohomology for the Galois group of a $p$-adic field $K$ with coefficients in a completed algebraic closure via vector bundles on the Hodge--Tate locus of the Cartier--Witt stack. This also gives a new perspective on classical Sen theory; for example it explains the appearance of an analogue of Colmez' period ring $B_{\Sen}$ in a geometric way. 
	\end{abstract}
	
	\maketitle
	\tableofcontents
\section{Introduction}
\label{sec:introduction}

Let $K$ be a $p$-adic field, i.e., a mixed characteristic $(0,p)$ complete discretely valued field with perfect residue field. 
Let $C=\widehat{\overline{K}}$ be the $p$-adic completion of an algebraic closure of $K$. 
The field $C$ comes equipped with a continuous action of the Galois group $G_K:=\mathrm{Gal}(\overline{K}/K)$ of $K$. 
Let $\mathrm{Rep}_C(G_K)$ be the category of semilinear continuous $G_K$-representations on finite dimensional $C$-vector spaces.
The goal of this article is to give a new answer to the following classical question:
\begin{question}\label{q:q-intro}
	How can we describe the category $\mathrm{Rep}_C(G_K)$ in terms of linear algebraic data?
\end{question}

Our first main result is a new description of this category in terms of the Hodge--Tate locus in the Cartier--Witt stack of Bhatt--Lurie (notations to be explained below):
 \begin{theorem}\label{t:main-thm-intro}
 	There is a natural equivalence of categories
\[\mathrm{Rep}_C(G_K)\to 2\text{-}\varinjlim\limits_{L/K} \Vec([\mathrm{Spf}(\mathcal{O}_L)^{\rm HT}/\Gal(L/K)])[\tfrac{1}{p}]\]
where $L|K$ ranges through finite Galois extensions of $K$ in $C$ and $\Vec$ is the category of vector bundles on the stack quotient $[\mathrm{Spf}(\mathcal{O}_L)^{\rm HT}/\Gal(L/K)]$.
\end{theorem}
\subsection{The perspective of Sen theory}
In order to motivate \Cref{t:main-thm-intro} and explain its relevance, let us first recall that 
the first approach to answer \Cref{q:q-intro} is provided by \textit{Sen theory}: 
 Let $K_\infty$ be the (uncompleted) cyclotomic $\Z_p$-extension of $K$. 
Sen (\cite{{sen1980continuous}}) constructs a functor 
\[ S:\mathrm{Rep}_C(G_K)\to \Big\{\begin{array}{@{}c@{}l}\text{finite dim.\ $K_\infty$-vector spaces $W$}\\ \text{with an endomorphism $\theta:W\to W$}\end{array}\Big\}.\]

This functor is not full, but on a certain full sub-category of ``small'' $C$-representations, it factors through a fully faithful functor $S_0$ into \textit{Sen modules} over $K$, which is the category of pairs $(M,\theta)$, where $M$ is a finite dimensional $K$-vector space and $\theta : M \to M$ an endomorphism. 
As any object in $\mathrm{Rep}_C(G_K)$ is in the domain of this functor after passing to a finite subextension $K\subseteq K_n\subseteq K_\infty$, this defines the functor $S$ in the colimit over $n$. 
One might try to  descend this construction back from $K_n$ to $K$, but it is not true that the functor preserves Galois descent data: the dependence on the embedding $K\subseteq C$ causes a subtle issue with functoriality of the construction.

Regarding essential surjectivity, if one tries to go into the other direction, from Sen modules to Galois representations, one can send a pair $(M,\theta)$ to the Galois representation on $V=M\otimes_K C$ with $\sigma \in G_K$ acting via
\[
\sigma(m \otimes c) = \exp(\log(\chi(\sigma))\cdot \theta)\cdot m \otimes \sigma(c)
\]
($\chi$ is the cyclotomic character) assuming that the series defining $\exp(\log(\chi(\sigma))\cdot \theta)$ converges for all $\sigma \in G_K$. 
This defines a fully faithful functor from Sen modules satisfying this convergence condition to $\mathrm{Rep}_C(G_K)$, which is a partial inverse to $S_0$. 
But this does not capture all Sen modules.

When the residue field $k$ of $K$ is algebraically closed, Sen shows that it is still possible to extend $S_0$ to an equivalence
\[ \mathrm{Rep}_C(G_K)\xrightarrow{\sim} \{\text{Sen modules}\}.\]
However, this equivalence is highly non-canonical as it depends on additional choices. 
For general $K$ (i.e.\ when $k$ is not algebraically closed, e.g.\ finite) the situation is even more subtle: one cannot expect even to have a fully faithful functor from $\mathrm{Rep}_C(G_K)$ to Sen modules, as  there are semilinear continuous $C$-representations $V$ of $G_K$ for which $\mathrm{End}(V)$ is a non-commutative division algebra (see \cite[Remark after Theorem~10]{sen1980continuous} or \cite[Remarque 2 after Th\'eor\`eme~2.14]{fontaine2004arithmetique}), but the endomorphism algebra $\mathrm{End}(M,\theta)$ of a Sen module can never be of this form.
But even in special cases where such a functor exists, one cannot expect to describe the essential image in a simple linear-algebraic way: for example, such a functor exists in the much simpler case of rank one, but we will show in \Cref{t:Picv} that it identifies the group of isomorphism classes of continuous semi-linear characters $G_K\to C^\times$ with the kernel of a certain canonical surjection $K\to \Br(K)[p^\infty]$.

In summary, $S$ allows us to characterise $C$-representations up to passage to a finite extension of $K$, or subject to a convergence condition, but in general does not describe the category $\Rep_C(G_K)$. 
It therefore stops short of giving a complete answer to \Cref{q:q-intro}. 
Indeed, the relation between $C$-linear representations and Sen modules is not as simple as it may seem at first sight. 

\subsection{The perspective of the $p$-adic Simpson correspondence}
To put the problem in perspective and analyze this issue, it is useful to recast the above setup in the much more general framework of $p$-adic non-abelian Hodge theory. 
Let $X$ be an adic space over $\mathrm{Spa}(\Q_p)$. 
The v-site of $X$ is the category of all perfectoid spaces $S$ over $X$ endowed with Scholze's v-topology (generated by surjective morphisms between affinoid perfectoid spaces over $X$). 
It has a structure sheaf $\mathcal{O}$, sending $S$ to $\mathcal{O}_S(S)$. 
The category of v-vector bundles $\mathrm{Vec}(X_v,\mathcal{O})$ is defined as the category of finite locally free sheaves of $\mathcal{O}$-modules on $X_v$. 
If $X=\mathrm{Spa}(K)$, then the choice of $K\subseteq C$ induces an equivalence $\mathrm{Vec}(X_v,\mathcal{O})=\mathrm{Rep}_C(G_K)$ to the category studied by Sen theory. 
In particular, \Cref{q:q-intro} can be recast as asking about the category of v-vector bundles on $\Spa(K)$.

If instead $X$ is smooth over a complete algebraically closed extension $L$ of $\Q_p$, then  $\mathrm{Vec}(X_v,\mathcal{O})$ is canonically equivalent to the category of \textit{generalized representations} on $X$ introduced by Faltings. 
It is a much richer category than the category of analytic vector bundles on $X$ (which, in the example $X=\mathrm{Spa}(K)$ is nothing but the category of finite dimensional $K$-vector spaces). 
In general, $p$-adic non-abelian Hodge theory seeks to describe the category $\mathrm{Vec}(X_v,\mathcal{O})$ in terms of \textit{Higgs bundles}: let $X$ be a smooth rigid space over a complete algebraically closed non-archimedean extension $L$ of $\Q_p$. 
Then a Higgs bundle is a pair $(\mathcal{E}, \theta)$, where $\mathcal{E}$ is an analytic vector bundle on $X$ and $\theta: \mathcal{E} \to \mathcal{E} \otimes_{\mathcal{O}_X} \Omega_X^1(-1)$ is an $\mathcal{O}_X$-linear map such that $\theta \wedge \theta=0$. 
In this setting, an instance of such a non-abelian Hodge correspondence is the ``local $p$-adic Simpson correspondence'', due to Faltings \cite{faltings2005p} and studied further by Abbes--Gros--Tsuji \cite{abbes2016p} and Wang \cite{wang2021p} among others: suppose that $X$  admits an \'etale map (a \textit{toric chart})
 \[
 c:X \to \mathbb{T}_L^n = \mathrm{Spa}(L\langle T_1^{\pm 1}, \dots, T_n^{\pm 1} \rangle) 
 \]
 to the $n$-dimensional rigid torus over $L$ for some $n\geq 1$. 
Then any such chart induces an equivalence 
 \[ LS_c:\{\text{small v-vector bundles}\}\xrightarrow{\sim}\{\text{small Higgs bundles}\}.\]
  Here smallness is a technical condition which means that the considered v-vector bundle or Higgs bundles are $p$-adically close to the unit object. 
Towards the goal of describing v-vector bundles on $X$,
 the local $p$-adic Simpson correspondence has two serious drawbacks. 
First, it only relates \textit{small} objects on both sides. 
As any smooth rigid space over $L$ admits a toric chart \'etale-locally, and any v-vector bundle or Higgs bundle becomes small \'etale-locally, one might try to get rid of this smallness assumption (and of the condition on $X$) by an appropriate gluing procedure. 
However, and this is the second issue, the equivalence of categories is \textit{not} functorial in $X$, but only in the pair $(X, c)$ formed by $X$ and the choice of a toric chart $c$. 
In fact, it is not true that one always has an equivalence between all v-vector bundles and all Higgs bundles on any smooth rigid space $X$ over $L$ (this fails already  for the $2$-dimensional unit ball, cf. 
\cite[\S6]{vlinebundles}). 
\medskip

Formulated in this way, this approach to describing v-vector bundles on smooth $X$ over $L$ is entirely parallel to the arithmetic situation of $X=\mathrm{Spa}(K)$ in terms of Sen theory, with Higgs bundles (resp.\ small Higgs bundles) playing the role of Sen modules (resp.\ Sen modules satisfying the above convergence condition)\footnote{To make the analogy even more apparent, note that if $\nu_\ast : \widetilde{X_v} \to \widetilde{X_{\rm et}}$ is the natural morphism of topoi, then for $X$ smooth over $L$, Scholze proved in \cite{scholze2013p} that there is an isomorphism $R^1\nu_\ast \mathcal{O} = \Omega_X^1(-1)$ , while for $X=\mathrm{Spa}(K)$ Tate's results on continuous cohomology \cite{tate1967p} imply that $R^1\nu_\ast \mathcal{O}=\O$ on $\Spa(K)$. 
So when Higgs bundles are defined in terms of $R^1\nu_{\ast}\O$, then Sen modules \text{are} Higgs bundles on $\Spa(K)$.}. 
In both cases, the statement in the ``small'' case is not functorial enough (it is only so after fixing a toric chart, or an embedding $K\to C$) to be globalized for the \'etale topology, and thus a complete description of v-vector bundles on $X$ via \'etale data is missing.

\subsection{A new perspective via the Hodge--Tate stack}

The starting point of this article is the idea that the recent introduction of ring stacks in $p$-adic Hodge theory by Drinfeld \cite{drinfeld2020prismatization} and Bhatt--Lurie  \cite{bhatt2022absolute}, \cite{bhatt2022prismatization} will help finding a more canonical formulation of such results and thereby shed new light on the local and global $p$-adic Simpson correspondences. 
The goal of this paper is to realize this concretely in the case $X=\mathrm{Spa}(K)$. 
Although this is quite a simple situation, it is already very much non-trivial as evidenced by the various subtleties of Sen theory, and we find it interesting that many features are already visible in this case. 
We regard this as a proof-of-concept for the general case.
\medskip

Let us make this strategy precise in the case of interest. 
We will only need a small part of the work of Drinfeld and Bhatt--Lurie; we use the notations of \cite{bhatt2022absolute}. 
To any $p$-complete ring $R$, Bhatt--Lurie attach an fpqc stack $\mathrm{Spf}(R)^{\rm HT}$, the \textit{Hodge--Tate stack of $\mathrm{Spf}(R)$}, on the category of (commutative) rings in which $p$ is nilpotent, with the property that, for $R$ quasi-syntomic, the category $\mathcal{D}(\mathrm{Spf}(R)^{\rm HT})$ of quasicoherent complexes on $\mathrm{Spf}(R)^{\rm HT}$ is naturally equivalent to the derived category of crystals of $\overline{\mathcal{O}}_\prism$-modules on the (absolute) prismatic site of $\mathrm{Spf}(R)$\footnote{Hence, (a relative variant of) the Hodge--Tate stack can be thought of as an analogue for Hodge--Tate cohomology of $p$-adic formal schemes of Simpson's de Rham stack for de Rham cohomology of complex varieties.}. 
If $R$ is perfectoid, $\mathrm{Spf}(R)^{\rm HT}$ is canonically isomorphic to $\mathrm{Spf}(R)$. 
This applies for example to $R=\mathcal{O}_C$. 
As a consequence of this isomorphism, the natural morphism $\mathrm{Spf}(\mathcal{O}_C)\to \mathrm{Spf}(\mathcal{O}_K)$ lifts to a Galois-equivariant morphism
\[
\mathrm{Spf}(\mathcal{O}_C) \to \mathrm{Spf}(\mathcal{O}_K)^{\rm HT}.
\]
Therefore, we obtain by pullback a canonical functor 
\[
\alst_K \colon \calPerf(\mathrm{Spf}(\mathcal{O}_K)^{\rm HT})[\tfrac{1}{p}]\to \calPerf(\Spa (K)_v)
\]
from the isogeny category of the category of perfect complexes on the Hodge--Tate stack of $\mathrm{Spf}(\mathcal{O}_K)$ towards v-perfect complexes on $\mathrm{Spa}(K)$. 

On the other hand, any choice of a  uniformizer $\pi$ of $K$ induces a Breuil--Kisin prism which induces a map
\[
\rho_\pi: \mathrm{Spf}(\mathcal{O}_K) \to \mathrm{Spf}(\mathcal{O}_K)^{\rm HT}
\]
that Bhatt--Lurie show is a faithfully flat torsor under a certain group $G_{\pi}$. Considering the tangent action induces on the pullback of any module along  $\rho_\pi$ a natural endomorphism $\Theta_\pi$. After inverting $\pi$, this defines a functor
\[
\beta_\pi: \calPerf(\mathrm{Spf}(\mathcal{O}_K)^{\rm HT})\tf\to \calPerf (K[\Theta_\pi]).
\]

The basic idea of our approach to non-abelian Hodge theory via the Cartier--Witt stack is now to study the diagram 
\[
\begin{tikzcd}
	& \calPerf(\mathrm{Spf}(\mathcal{O}_K)^{\rm HT}
	)\tf \arrow[ld, "\alst_K"'] \arrow[rd, "\beta_\pi"] &                      \\
	\calPerf(\Spa(K)_v) \arrow[rr, "\Sen", dotted] &                                                                    & {\calPerf(K[\Theta_\pi])}
\end{tikzcd}\]
in which we note that the left map is completely canonical, whereas the right map depends on $\pi$ even though neither the target nor source category do. Roughly, the idea is now that taking a preimage under $\alst$ and mapping it down under $\beta_\pi$ should give an analogue of Sen's functor\footnote{More precisely, it is an analogue for Sen's functor  defined using the Kummer tower of $\pi$ rather than the cyclotomic tower. These two constructions are equivalent up to a non-canonical transformation.}. However, as is to be expected from the above discussion of Sen theory, the morphism $\alst_K$ is not essentially surjective, reflecting the fact that $\Sen$ is only partially defined. 

Instead, we have the following description of the above diagram, which is our second main result:
\begin{theorem}[{\Cref{sec:c_k-semil-galo-1-pullback-fully-faithful-up-to-quasi-isogeny}, \Cref{sec:c_k-semil-galo-3-description-of-essential-image}, ~\Cref{sec:an-analytic-variant-corollary-analytic-variant}}]
\label{main-theorem-introduction}
Let $E$ be the minimal polynomial of $\pi$ over $K_0$, the maximal unramified subfield of $K$.
Let $\delta_{\mathcal{O}_K/\Z_p}=(E'(\pi))$ be the different of $\mathcal{O}_K|\Z_p$. 
\begin{enumerate}
\item The functor $\alst_K$ is fully faithful. Its essential image is the bounded derived category of \textit{nearly Hodge--Tate representations}, i.e.\ semilinear continuous $C$-representations of $G_K$ whose Sen operator has all its eigenvalues in $\Z+ \delta_{\mathcal{O}_K/\Z_p}^{-1}\cdot \mathfrak{m}_{\overline{K}}$. 
\item The functor $\beta_\pi$ is fully faithful. Its essential image consists of those perfect complexes $M$ of $K[\Theta_\pi]$-modules for which $H^\ast(M)$ is finite dimensional over $K$ and $\Theta_\pi^p-E'(\pi)^{p-1}\Theta_\pi$ acts topologically nilpotently on $H^\ast(M)$. 
\end{enumerate} 
\end{theorem}

Equivalently, and in close parallel to (1), condition (2) means that on each cohomology group, $E'(\pi)^{-1}\Theta_\pi$ has all generalised eigenvalues in  $\Z+ \delta_{\mathcal{O}_K/\Z_p}^{-1}\cdot \mathfrak{m}_{\overline{K}}$.

Moreover, in \Cref{sec:complexes-Hodge--Tate-1-complexes-on-ht-locus-for-r} we describe $\mathcal{D}(\Spf(\O_K)^\HT)$ in terms of complexes of $\calO_K[\Theta_\pi]$-modules. 
\begin{remark}
	\label{relation-min-wang-gao}
	Related results were proved recently by Min--Wang \cite{min2021hodge}\cite{MinWang22} and Gao \cite{gao2022hodge} in terms of Hodge--Tate prismatic crystals, which are equivalent to vector bundles on $\Spf(\O_K)^{\HT}$ by \cite[Proposition 8.15]{bhatt2022prismatization}. Their results only hold at the abelian level (i.e.\ for vector bundles instead of perfect complexes), but work more generally for smooth rigid spaces of good reduction over a $p$-adic field  (instead of just $X=\mathrm{Spa}(K)$), also using the prismatic formalism. 
	While they use the prismatic site, we adopt the perspective furnished by Drinfeld and Bhatt--Lurie's stacks. For (2) we basically follow Bhatt--Lurie's arguments for $K=\Q_p$.
\end{remark}

It follows that any object is in the essential image of $\beta_\pi$ after passing to a finite extension $L|K$.
For \Cref{q:q-intro}, the crucial point is now that since $\alst_K$ is completely canonical and functorial, the subtleties that keep us from descending Sen's functor disappear for $\alst_K$. We deduce:
\begin{theorem}\label{t:thm-3-intro}
If $L|K$ is a finite Galois extension, the functor $\alst_L$ induces a fully faithful functor
 \[ \alst_{L/K} \colon \calPerf([\mathrm{Spf}(\mathcal{O}_L)^{\rm HT}/{\Gal(L/K)}])[\tfrac{1}{p}] \to \calPerf(\Spa(K)_v)
 \]
  and any object in $\calPerf(\Spa(K)_v)$ lies in the image of $\alst_{L/K}$ for some $L$. 
\end{theorem}
This allows us to compute Galois cohomology of semilinear $G_K$-representations, i.e.\ cohomology of vector bundles on $\Spa(K)_v$: 
Let $\mathcal E\in  \calPerf(\mathrm{Spf}(\mathcal{O}_K)^{\rm HT}
)$ and let $V:=\alst_K(\mathcal E)$, $(M,\theta_M):=\beta_\pi(\mathcal E)$. Then
\[R\Gamma(X_v,V)=R\Gamma(\mathrm{Spf}(\mathcal{O}_K)^{\rm HT},\mathcal{E})[\tfrac{1}{p}]=\mathrm{fib}(M\xrightarrow{\theta_M}M).\]

\Cref{t:thm-3-intro} is an easy consequence of \Cref{main-theorem-introduction}.(1). 
The proof of \Cref{main-theorem-introduction} relies crucially on the geometry of $\mathrm{Spf}(\mathcal{O}_K)^{\rm HT}$. Namely, Bhatt--Lurie show that the map $\rho_\pi:\Spf(\O_K)\to \mathrm{Spf}(\mathcal{O}_K)^{\rm HT}$
realizes $\mathrm{Spf}(\mathcal{O}_K)^{\rm HT}$ as the (relative) classifying stack of an explicit group $G_\pi$ (which is either isomorphic to $\mathbb{G}^\sharp_m$ or to $\Ga^\sharp$). 
This gives a concrete description of perfect complexes on this stack, leading to \Cref{main-theorem-introduction}.(2); in particular, let us point out that the (seemingly strange) nearly Hodge--Tate condition appears very naturally through the description of the Cartier dual of $G_\pi$. 
Using this presentation, one can also unravel what the functor $\alst_K$ is doing in terms of the description provided by (2): let $Z_\pi \to \mathrm{Spf}(\mathcal{O}_C)$ be the base change of $\rho_\pi$ along the map $\Spf(\mathcal{O}_C) \to \mathrm{Spf}(\mathcal{O}_K)^{\rm HT}$ mentioned above.
 Let $A_\en = \mathcal{O}(Z_\pi)$ be the ring of functions on $Z_\pi$ and let $B_\en = A_\en[1/p]$. 
The ring $B_\en$ is endowed with an endomorphism $\Theta_\pi$ and a commuting continuous $C$-semilinear $G_K$-action. 
Then $\alst_K$ sends a complex $\mathcal E$ on $\mathrm{Spf}(\mathcal{O}_K)^{\rm HT}$ with $\beta_\pi(\mathcal E)=(M,\theta_M)$ to
\[
\mathrm{fib}\big(M \otimes_{K} B_\en  \xrightarrow{1 \otimes \Theta_\pi + \theta_M \otimes 1}  M \otimes_{K} B_\en \big).
\]
In other words, the ring $B_\en$ functions as a period ring in this context. 
This ring is closely related to Colmez' ring $B_{\rm Sen}$ (introduced by him in \cite{colmez1994resultat} to reformulate the construction of the Sen functor in the style of Fontaine) but is different and has not been considered earlier in the literature as far as we know. 
(Interestingly, it seems to be more easily related to a variant of Sen theory using the Kummer tower, rather than the cyclotomic tower as is usually done. 
Compare \cite{gao2022hodge}.)

Let us remark at this point that our proof of the essential surjectivity aspect in \Cref{t:thm-3-intro}.(1)  uses Sen theory: we have nothing new to say regarding the (key) \textit{decompletion} process in the theory. 
But we stress that $\alst_K$ is a canonical, geometrically defined, functor (which after a choice becomes identified with an explicit Fontaine-type functor). 
This is the main point of the Hodge--Tate stack approach and we believe that this point of view will give a fruitful new perspective on $p$-adic non-abelian Hodge theory in general. In work in progress, we investigate the case of smooth varieties over algebraically closed non-archimedean fields and $p$-adic fields by the same methods.

One might ask if one can use this description to obtain a more explicit description of the category $\Rep_C(G_K)$ in terms of linear algebra data. In order to explore what such a description could look like, we investigate in the last section the rank one case. In dimension $1$, Sen's division algebra counterexample discussed at the beginning of the introduction does not apply and one might hope to extend the fully faithful functor of point (1) of \Cref{main-theorem-introduction} to an equivalence between continuous semilinear $C$-representations of $G_K$ and rank one Sen modules. 
We prove that this is not possible, by showing that there are more objects on the side of Sen modules than on the Galois side. Rather, there is in general a canonical short exact sequence
\[ 0\to \Pic_v(K)\to K\to \Br(K)[p^\infty]\to 0\]
where $\Br(K)=H^2_{\et}(K,\Gm)$ is the Brauer group.
This means that already in the very simple case of rank 1, the category of representations of $G_K$ sees subtle arithmetic information of the $p$-adic field $K$ like the $p$-primary part of the Brauer group (which we recall is related to local class field theory), which are therefore captured by the Hodge--Tate stacks in \Cref{t:main-thm-intro}.

\subsection{Plan of the paper} After some brief recollection on the Hodge--Tate stack (the Hodge--Tate locus in Bhatt--Lurie's Cartier--Witt stack), in \Cref{sec:recollection-ht-cw-stack}, we prove (a more general version of) point (2) of \Cref{main-theorem-introduction} in \Cref{sec:complexes-Hodge--Tate-for-o-k}. 
\Cref{sec:galo-acti-cart} contains some explicit computations describing the Galois action on $B_\en$ (in particular a short discussion of its relation to Colmez' $B_{\rm Sen}$) and its cohomology, which allow us to complete the proof of \Cref{main-theorem-introduction} in \Cref{sec:c_k-semil-galo}. Finally, in \Cref{sec:v-picard-group-of-local-fields} we compute the v-Picard group of $p$-adic fields.

\subsection*{Acknowledgments}

 The authors would like to thank Bhargav Bhatt, Yu Min, Juan Esteban Rodr\'iguez Camargo and Lue Pan for helpful conversations on Sen theory and feedback on a preliminary version of this paper, and Peter Scholze and Matti W\"urthen for useful discussions. We thank Jo\~ao Louren{\c{c}}o for pointing out a mistake in an earlier version. One author (J.A.) wants to thank the organizers of the conference in Darmstadt in October 2022 for the possibility to present the results of this paper. The authors want to thank the referee for their careful comments and their positive feedback.

The second author was funded by Deutsche Forschungsgemeinschaft (DFG, German Research Foundation) under Germany's Excellence Strategy -- EXC-2047/1 -- 390685813, as well as by DFG Project-ID 444845124 -- TRR 326, and was moreover supported by DFG via the Leibniz-Preis of Peter Scholze.

\subsection*{Notations}
\label{sec:notations}

We will use the following notations.
\begin{enumerate}
\item $p$ is a prime,
\item $K$ is a $p$-adic field (complete discretely valued with perfect residue field) with ring of integers $\mathcal{O}_K$ and residue field $k$.
  \item $\Nilp_p$ is the category of (classical, commutative) rings $R$ such that $p$ is nilpotent in $R$,
  \item If $X$ is a $p$-adic formal scheme, we denote by $X^{\prism}$ (resp. $X^{\rm HT}$) the Cartier--Witt stack of $X$ (resp. the Hodge--Tate locus in the Cartier--Witt stack of $X$, or Hodge--Tate stack of $X$), which is denoted $\WCart_X$ (resp. $\WCart_X^{\rm HT}$) in \cite{bhatt2022absolute}, \cite{bhatt2022prismatization}.
\end{enumerate}

\section{Complexes on the Hodge--Tate stack}
\subsection{Recollections on the Hodge--Tate stack for $\mathcal{O}_K$}
\label{sec:recollection-ht-cw-stack}
The \textit{Cartier--Witt stack}
$$\WCart$$ is the groupoid-valued functor on $\mathrm{Nilp}_p$ sending $R$ to the groupoid of \textit{Cartier--Witt divisors}, i.e. $W(R)$-linear maps $\alpha: I \to W(R)$ from an invertible $W(R)$-module $I$ satisfying the conditions that the image of the composite $I \overset{\alpha} \to W(R) \to R$ is a nilpotent ideal of $R$ and the image of the composite $I \overset{\alpha} \to W(R) \overset{\delta} \to W(R)$ generates the unit ideal. Since $p$ is nilpotent in $R$, Zariski-locally on $\mathrm{Spec}(R)$, $I$ is principal generated by some element $d$ and the above two conditions can be formulated concretely by saying that
$$
\alpha(d)=\sum_n V^n [r_n],
$$
 with $r_0$ nilpotent and $r_1 \in R^\times$. The functor $\WCart$ in fact defines an fpqc stack on $\mathrm{Nilp}_p$. The \textit{Hodge--Tate locus}
 $$\WCart^{\rm HT}$$ in the Cartier--Witt stack is defined to be the closed substack defined by the condition that, for $R \in \mathrm{Nilp}_p$, a Cartier--Witt divisor $\alpha: I \to W(R)$ belongs to the full subcategory $\WCart^{\rm HT}(R)$ of $\WCart(R)$ if the composite $I \overset{\alpha} \to W(R) \to R$ is zero (equivalently, in terms of the above explicit formula, if $r_0=0$). 
 
 If $X$ is a bounded $p$-adic formal scheme, we define the \textit{Cartier--Witt stack of $X$}, or \textit{prismatization of $X$}, to be the groupoid-valued functor $$X^\prism$$ on $\mathrm{Nilp}_p$ sending $R$ to the groupoid of pairs $(\alpha: I \to W(R), \eta: \mathrm{Spec}(\overline{W(R)}:=\mathrm{cone}(\alpha)) \to X)$, where $(\alpha: I \to W(R)) \in \WCart(R)$ is a Cartier--Witt divisor and $\eta$ is a map of derived formal schemes. The \textit{Hodge--Tate stack of $X$}, which will be the main player in this text, is defined as the fiber product 
 \[
  \xymatrix{
    X^{\rm HT}  \ar[r] \ar[d] & X^\prism \ar[d] \\
   \WCart^{\rm HT} \ar[r] & \WCart.
  }
\]
For $R\in \mathrm{Nilp}_p$, if $(\alpha: I \to W(R)) \in \WCart^{\rm HT}(R)$, the map $\alpha$ factors through $VW(R)$ and hence one gets a map $\overline{W(R)} \to W(R)/VW(R)=R$. Therefore any point $(\alpha: I \to W(R), \eta: \mathrm{Spec}(\overline{W(R)}) \to X) \in X^{\rm HT}(R)$ gives rise to a map $\mathrm{Spec}(R) \to \mathrm{Spec}(\overline{W(R)}) \overset{\eta} \to X$. This defines a morphism $X^{\rm HT} \to X$, called the Hodge--Tate structure morphism.

\begin{remark}
Note that we have $\WCart=\mathrm{Spf}(\Z_p)^\prism$ and $\WCart^{\rm HT}=\mathrm{Spf}(\Z_p)^{\rm HT}$, as $\mathrm{Spf}(\Z_p)$ is the terminal object in the category of $p$-adic formal schemes.
\end{remark}

A fundamental feature of these stacks is their relation to prisms, which is given by the following simple construction (\cite[Construction 3.10]{bhatt2022prismatization}). Let $X$ be a bounded $p$-adic formal scheme and $(A,I)\in (X)_\prism$ an object of the absolute prismatic site of $X$. If $R$ is an $A$-algebra in $\mathrm{Nilp}_p$, the $A$-algebra structure on $R$ uniquely lifts to a $\delta$-$A$-algebra structure on $W(R)$. Base changing the inclusion $I \subset A$ along $A\to W(R)$ gives a Cartier--Witt divisor $\alpha: I \otimes_A W(R) \to W(R)$ together with a map $\eta: \mathrm{Spf}(\overline{W(R)}) \to \mathrm{Spf}(\overline{A}:=A/I) \to X$ (where the last map comes from the assumption that $(A,I) \in (X)_\prism$), hence a point in $X^\prism(R)$. Therefore, we get a map
$$
\rho_{X,A}: \mathrm{Spf}(A) \to X^\prism
$$  
($A$ is endowed with the $(p,I)$-adic topology). It induces a map
$$
\rho_{X,\overline{A}}: \mathrm{Spf}(\overline{A}) \to X^{\rm HT}.
$$  
   
Let us make this more explicit when $X= \mathrm{Spf}(\mathcal{O}_K)$. Let $\pi\in \mathcal{O}_K$ be a uniformizer. 
This choice yields the Breuil--Kisin prism
\[
  (A_\pi,I_\pi):=(W(k)[[u]],(E(u)))
\]
with $\overline{A_\pi}:=A_\pi/I_\pi\cong \mathcal{O}_K,\ u\mapsto \pi$, where $E(u)\in W(k)[u]$ is the minimal polynomial of $\pi\in K$ over $K_0:=W(k)[1/p]$, and the $\delta$-algebra structure is defined by $\delta(u)=0$, meaning $\varphi(u)=u^p$.
Set\footnote{We note that usually $e$ denotes the absolute ramification index of $K$, i.e., the degree of $E(u)$. 
As this will not be relevant to us, we use the symbol $e$ for something else.}
\[
  e:=E^\prime(\pi)\in \mathcal{O}_K.
\]
We let
\[
  \rho_\pi = \rho_{\mathrm{Spf}(\mathcal{O}_K), \overline{A_\pi}}\colon \Spf(\mathcal{O}_K)\to \mathrm{Spf}(\mathcal{O}_K)^{\rm HT}
\]
be the map induced by $(A_\pi,I_\pi)$ introduced above. 
Concretely, an $\mathcal{O}_K$-algebra $S$ is mapped to the image along $\mathrm{Spf}(\mathcal{O}_K)^{\rm HT}(\mathcal{O}_K)\to \mathrm{Spf}(\mathcal{O}_K)^{\rm HT}(S)$ of the point
\[
  (I_\pi\otimes_{A_\pi} W(\mathcal{O}_K)\to W(\mathcal{O}_K), \mathcal{O}_K\cong A_\pi/I_\pi\to \overline{W(\mathcal{O}_K)}).
\]
\begin{definition}Let $G_\pi$ be the group sheaf of automorphisms of the point $\rho_\pi$, i.e.\ $G_\pi$ is the functor
\[
  S\mapsto \mathrm{Aut}_{\mathrm{Spf}(\mathcal{O}_K)^{\rm HT}(S)}(\Spf(S)\to \Spf(\mathcal{O}_K)\xrightarrow{\rho_\pi} \mathrm{Spf}(\mathcal{O}_K)^{\rm HT})
\]
on $p$-complete $\mathcal{O}_K$-algebras. 
\end{definition}
By \cite[Example 9.6]{bhatt2022prismatization} the group $G_\pi$ identifies concretely with the subgroup
\[
  G_\pi=\{(t,a)\in \mathbb{G}_m^\sharp\ltimes \mathbb{G}_a^\sharp\ |\ t-1=e\cdot a\}
\]
(see \cite[\S3.4]{bhatt2022absolute} for the definition of $\mathbb{G}_m^\sharp$ and $\mathbb{G}_a^\sharp$).
The action of $\mathbb{G}_m^\sharp(S)\cong W^\times[F](S)$ on $\rho_\pi$ is via the natural $W(S)$-multiplication on $I_\pi\otimes_{A_\pi} W(S)$ while the action of $\mathbb{G}_a^\sharp(S)\cong W[F](S)$ is via homotopies on the morphism $\mathcal{O}_K\to \overline{W(S)}$ of animated rings, cf.\ \cite[Construction 9.4]{bhatt2022prismatization}. 
It is useful to note that if $S$ is $p$-torsion free, then the action of an element $(t,a)\in G_\pi(S)$ is uniquely determined by its action via the projection $G_\pi\to \mathbb{G}_m^\sharp\cong W^\times[F]$ on $I_\pi\otimes_{A_\pi} W(S)$. 
Moreover, under this torsion freeness assumption, an element $(a_0, a_1, \ldots )\in W^\times[F](S)\subseteq W(S)$ is uniquely determined by $a_0$. We now have:
\begin{proposition}[{\cite[Proposition~9.5]{bhatt2022prismatization}}]
The action of $G_{\pi}$ makes $\rho_{\pi}:\Spf(\O_K)\to \Spf(\mathcal{O}_K)^\HT$ into a $G_{\pi}$-torsor for the fpqc-topology. In particular, $\rho_{\pi}$ is affine and faithfully flat and induces an isomorphism  $\Spf(\mathcal{O}_K)^\HT=BG_\pi$.
\end{proposition}

To describe $G_\pi$ even more concretely, note that the projection $(t,a)\to a$ yields an isomorphism 
\[G_\pi\cong \mathbb{G}_a^\sharp=\Spf(\mathcal{O}_K[\frac{a^n}{n!}\ |\ n\geq 0]^{\wedge}_p)\]
 of formal schemes, such that the comultiplication on $\mathcal{O}_{G_\pi}$ identifies with the formal group law
\[
  a+b+e\cdot a b.
\]
If $e\in \mathcal{O}_K^\times$, then $G_\pi\cong \mathbb{G}_m^\sharp$ via the first projection, and if $e\in \mathfrak{m}_{K}:=\pi\cdot \mathcal{O}_K$, then $G_\pi\cong \mathbb{G}_a^\sharp$ given by the maps
\[
  G_\pi\to \mathbb{G}_a^\sharp,\ a\mapsto \frac{\log(ea+1)}{e}=\sum\limits_{n\geq 1}(-e)^{n-1}\frac{a^n}{n}
\]
and
\[
  \mathbb{G}_a^\sharp\to G_\pi,\ a\mapsto \frac{\mathrm{exp}(ea)-1}{e}=\sum\limits_{n\geq 1} e^{n-1}\frac{a^n}{n!}.
\]

\subsection{From complexes on the Hodge--Tate stack to Sen modules}
\label{sec:complexes-Hodge--Tate-for-o-k}
We keep the notations of the previous section. 
In this section, following very closely an argument of \cite{bhatt2022absolute},  we want to give an explicit description of the $\infty$-category of quasicoherent complexes
\[
\mathcal{D}(\mathrm{Spf}(\mathcal{O}_K)^{\rm HT})
\]
on the Hodge--Tate stack $\mathrm{Spf}(\mathcal{O}_K)^{\rm HT}$ of $\mathcal{O}_K$ (cf. 
\cite{bhatt2022absolute}[Definition 3.5.1]). 
We stress that this description is \textit{not} canonical but involves the choice of a uniformizer $\pi \in \mathcal{O}_K$.

The choice of such a $\pi \in \mathcal{O}_K$ yields a faithfully flat morphism
\[
  \rho_\pi\colon \Spf(\mathcal{O}_K)\to \mathrm{Spf}(\mathcal{O}_K)^{\rm HT},
\]
whose automorphism group sheaf is $G_\pi$ as in \Cref{sec:recollection-ht-cw-stack}.
As in \cite[Construction 3.5.4]{bhatt2022absolute} we can construct an automorphism of the projection
\[
  \mathrm{Spf}(\mathcal{O}_K)^{\rm HT}\times_{\Spf(\mathcal{O}_K)} \Spf(\mathcal{O}_K[\varepsilon]/(\varepsilon^2))\to \mathrm{Spf}(\mathcal{O}_K)^{\rm HT}
\]
yielding a Sen operator for $\mathcal{O}_K$. 
Namely, for $\mathcal{E}\in \mathcal{D}(\mathrm{Spf}(\mathcal{O}_K)^{\rm HT})$ multiplication by the element\footnote{ Here all necessary higher divided powers of $\varepsilon, 1+e\varepsilon$ are defined to be $0$.} $(1+e\varepsilon,\varepsilon)\in G_\pi(R[\varepsilon]/(\varepsilon^2))$ yields an automorphism of 
\[
  \mathcal{E}\otimes_{\mathcal{O}_K} \mathcal{O}_K[\varepsilon]/(\varepsilon^2),
\]
which can be interpreted as a morphism $\mathrm{Id}+\varepsilon \Theta_{\pi,\mathcal{E}}\colon \mathcal{E}\to \mathcal{E}\otimes_{\mathcal{O}_K} \mathcal{O}_K[\varepsilon]/(\varepsilon^2)$ for some endomorphism $\Theta_{\pi,\mathcal{E}}\colon \mathcal{E}\to \mathcal{E}$ in $\mathcal{D}(\mathrm{Spf}(\mathcal{O}_K)^{\rm HT})$. 
The endomorphism $\Theta_{\pi,\mathcal{E}}$ is called the \textit{Sen operator} of $\mathcal{E}$.
By a slight abuse of notation, we denote by $\Theta_{\pi,\mathcal{E}}$ also the pullback of $\Theta_{\pi,\mathcal{E}}$ along the map $\rho_\pi$.
 The construction of $\Theta_{\pi,\mathcal E}$ is clearly functorial in $\mathcal{E}$, thus pullback along $\rho_\pi$  defines a natural functor
 \[
\beta_{\pi}^+:\mathcal{D}(\mathrm{Spf}(\mathcal{O}_K)^{\rm HT})\to \mathcal{D}(\mathcal{O}_K[\Theta_\pi]),\ \mathcal{E}\mapsto (\rho_\pi^\ast \mathcal{E}, \Theta_{\pi,\mathcal{E}}).
\]

\begin{example}
  \label{sec:complexes-Hodge--Tate-example-sen-for-k-vs-sen-for-z-p}
  The natural map $h\colon \mathrm{Spf}(\mathcal{O}_K)^{\rm HT}\to \mathrm{Spf}(\Z_p)^{\rm HT}\cong B\mathbb{G}_m^\sharp$ is induced by the morphism of groups
  \[
    G_\pi\to \mathbb{G}_m^\sharp,\ (t,a)\mapsto t.
  \]
  From the constructions of the Sen operators for $\mathcal{O}_K$ and $\Z_p$, respectively, we can conclude that if $\mathcal{E}\in \mathcal{D}(\mathrm{Spf}(\Z_p)^{\rm HT})$ with Sen operator $\Theta_{\mathcal{E}}\colon \mathcal{E}\to \mathcal{E}$, then $h^\ast \mathcal{E}$ has Sen operator $\Theta_\pi=e\cdot h^\ast \Theta$. 
In particular, for the Breuil--Kisin line bundles $\mathcal{O}_{\mathrm{Spf}(\mathcal{O}_K)^{\rm HT}}\{n\}$ from \cite[Example 3.3.8]{bhatt2022absolute}, this means that $\mathcal{O}_{\mathrm{Spf}(\Z_p)^{\rm HT}}\{n\}$ pulls back to
  \[
    \mathcal{O}_{\mathrm{Spf}(\mathcal{O}_K)^{\rm HT}}\{n\}:=h^\ast \mathcal{O}_{\mathrm{Spf}(\Z_p)^{\rm HT}}\{n\},
  \]
  whose associated Sen operator is given by $e\cdot n$, cf.\ \cite[Example 3.5.6]{bhatt2022absolute}. 
We stress that $\mathcal{O}_{\mathrm{Spf}(\mathcal{O}_K)^{\rm HT}}\{n\}$ is canonically defined, but that its associated Sen operator depends on the choice of $\pi$ because $e$ depends on $\pi$.
\end{example}

As in \cite[Example 3.5.5]{bhatt2022absolute} the Sen operator satisfies the Leibniz rule for tensor products in $\mathcal{D}(\mathrm{Spf}(\mathcal{O}_K)^{\rm HT})$.

The main aim of this section is to prove the following analog of \cite[Theorem 3.5.8]{bhatt2022absolute}.

\begin{theorem}\label{t:beta}
  \label{sec:complexes-Hodge--Tate-1-complexes-on-ht-locus-for-r}
  The functor
  \[
    \beta_{\pi}^+:\mathcal{D}(\mathrm{Spf}(\mathcal{O}_K)^{\rm HT})\to \mathcal{D}(\mathcal{O}_K[\Theta_\pi]),\ \mathcal{E}\mapsto (\rho_\pi^\ast \mathcal{E}, \Theta_{\pi,\mathcal{E}})
  \]
  is fully faithful and its essential image consists of complexes $M\in \mathcal{D}(\mathcal{O}_K[\Theta_\pi])$ which are (derived) $\pi$-adically complete and such that the action of $\Theta^p_\pi-e^{p-1}\Theta_\pi$ on the cohomology $H^\ast(k\otimes_{\mathcal{O}_K}^L M)$ is locally nilpotent.
\end{theorem}

We note that the functor depends on the choice of the uniformizer $\pi\in \mathcal{O}_K$.
If $e\in \mathcal{O}_K$ is not a unit, then $\Theta_\pi^p-e^{p-1}\Theta_\pi=\Theta^p_\pi$ on $H^\ast(k\otimes_{\mathcal{O}_K}^L M)$, which is in accordance with the description of representations of $\mathbb{G}_a^\sharp$ as finite free $\mathcal{O}_K$-modules with a topologically nilpotent endomorphism.

For the proof of \Cref{t:beta} we will follow the arguments of \cite[Theorem 3.5.8]{bhatt2022absolute}.

\begin{lemma}
  \label{sec:complexes-Hodge--Tate-sen-operator-for-regular-representation}
  Set $\mathcal{E}:=\rho_{\pi,\ast}\mathcal{O}_{\Spf(\mathcal{O}_K)}$. 
Then
  \begin{enumerate}
  \item $\rho_\pi^\ast \mathcal{E}\cong \mathcal{O}_{G_\pi}=\widehat{\bigoplus\limits_{n\geq 0}} \mathcal{O}_K\cdot \frac{a^n}{n!}$ (where on the right we write for simplicity $\mathcal{O}_K$ for what should be denoted $\mathcal{O}_{\mathrm{Spf}(\mathcal{O}_K)}$),
  \item $\Theta_{\pi,\mathcal{E}}=(1+ea)\frac{\partial}{\partial a}$,
  \item the sequence $0\to \mathcal{O}_K\to \rho_\pi^\ast \mathcal{E}\xrightarrow{\Theta_{\pi,\mathcal{E}}} \rho_\pi^\ast\mathcal{E}\to 0$ is exact.
  \end{enumerate}
\end{lemma}
\begin{proof}
The isomorphism $\rho_\pi^\ast \mathcal{E}\cong \mathcal{O}_{G_\pi}$ follows from the projection formula. 
The description of $\mathcal{O}_{G_\pi}$ in terms of $\widehat{\bigoplus\limits_{n\geq 0}} \mathcal{O}_K\cdot \frac{a^n}{n!}$ follows if we identify $G_\pi\cong \mathbb{G}_a^\sharp$ (as formal schemes) via the map $(t,a)\to a$. 
Under the isomorphism $G_\pi\cong \mathbb{G}_a^\sharp,$ which transfers the comultiplication to the map 
\[
c\colon \widehat{\mathcal{O}_{\mathbb{G}_a^\sharp}}\to \widehat{\mathcal{O}_{\mathbb{G}_a^\sharp}}\widehat{\otimes}_{\mathcal{O}_K} \widehat{\mathcal{O}_{\mathbb{G}_a^\sharp}},\ a\mapsto a+b+e\cdot ab,
\] 
the Sen operator $\Theta_{\pi,\mathcal{E}}$ is constructed using the action of the element $\varepsilon\in \mathbb{G}_a^\sharp(\mathcal{O}_K[\varepsilon]/(\varepsilon^2))$. 
In other words, we have to look at the composition
\[
  \widehat{\mathcal{O}_{\mathbb{G}_a^\sharp}}\xrightarrow{c} \widehat{\mathcal{O}_{\mathbb{G}_a^\sharp}}\otimes_{\mathcal{O}_K} \widehat{\mathcal{O}_{\mathbb{G}_a^\sharp}}\xrightarrow{\mathrm{Id}\otimes(a\mapsto \varepsilon)} \widehat{\mathcal{O}_{\mathbb{G}_a^\sharp}}\otimes_{\mathcal{O}_K} \mathcal{O}_K[\varepsilon]/(\varepsilon^2),
\]
which sends some $f(a)\in \widehat{\mathcal{O}_{\mathbb{G}_a^\sharp}}$ to the element $f(a+(1+ea)\varepsilon)=f(a)+\varepsilon (1+ea)\frac{\partial f}{\partial a}(a)$ as desired.
Let $f(a)=\sum\limits_{n\geq 0} c_n\frac{a^n}{n!}\in \widehat{\mathcal{O}_{\mathbb{G}_a^\sharp}}$.
Then
\[
  \Theta_{\pi,\mathcal{E}}(f)=\sum\limits_{n\geq 1} c_n(en\frac{a^n}{n!}+\frac{a^{n-1}}{(n-1)!})=\sum\limits_{n\geq 0}(c_{n+1}+e n c_n)\frac{a^n}{n!}.
\]
Thus, if $\Theta_{\pi,\mathcal{E}}(f)=0$, then $c_{n+1}=-en c_n$ for all $n\geq 0$, which implies $c_{n+1}=0$ for all $n\geq 0$, i.e., $f\in \mathcal{O}_K$. 
Given an element $g(a)=\sum\limits_{n\geq 0} b_n\frac{a^n}{n!}$, then we can inductively solve the system $b_n=c_{n+1}+en c_n$ by starting with
\[\begin{matrix}
  c_0:=0 \\
  c_1:=b_0 \\
  c_2:=b_1-e\cdot 1\cdot c_1=b_1-e b_0\\
  \ldots \\
  c_{n+1}:=b_n-en c_n=\ldots=b_n-\ldots -(-e)^{n-1}n!b_0.
\end{matrix}\]
We see that if for the $p$-adic topology $b_n\to 0, n\to \infty$, then also $c_n\to 0, n\to \infty$. 
This finishes the proof of the lemma.
\end{proof}

Similarly to \cite[Proposition 3.5.11]{bhatt2022absolute} we can draw the following consequence.

\begin{proposition}
  \label{sec:complexes-Hodge--Tate-cohomology-on-Cartier--Witt-stack}
  For any $\mathcal{E}\in \mathcal{D}(\mathrm{Spf}(\mathcal{O}_K)^{\rm HT})$ there exists a canonical fiber sequence
  \[
    R\Gamma(\mathrm{Spf}(\mathcal{O}_K)^{\rm HT},\mathcal{E})\to \beta_\pi^+(\mathcal{E})\xrightarrow{\Theta_{\pi,\mathcal{E}}} \beta_\pi^+(\mathcal{E}).
  \]
  In particular, the functor $R\Gamma(\mathrm{Spf}(\mathcal{O}_K)^{\rm HT},-)$ commutes with all colimits/limits.
\end{proposition}
\begin{proof}
  Note that the Sen operator for $\mathcal{O}_{\mathrm{Spf}(\mathcal{O}_K)^{\rm HT}}$ is $0$, e.g., by \Cref{sec:complexes-Hodge--Tate-example-sen-for-k-vs-sen-for-z-p}. 
This implies (by naturality of the Sen operator) that the natural map $\mathcal{O}_{\mathrm{Spf}(\mathcal{O}_K)^{\rm HT}}\to \rho_{\pi,\ast}\rho_\pi^\ast\mathcal{O}_{\mathrm{Spf}(\mathcal{O}_K)^{\rm HT}}=\rho_{\pi,\ast}\mathcal{O}_{\Spf(\mathcal{O}_K)}$ induces a natural map
  \[
    \mathcal{O}_{\mathrm{Spf}(\mathcal{O}_K)^{\rm HT}}\to \mathrm{fib}(\rho_{\pi,\ast}(\mathcal{O}_{\Spf(\mathcal{O}_K)})\xrightarrow{\Theta_{\pi,\rho_{\pi,\ast}\mathcal{O}_{\Spf(\mathcal{O}_K)}}} \rho_{\pi,\ast}(\mathcal{O}_{\Spf(\mathcal{O}_K)})),
  \]
  which by faithfully flat descent along $\rho_\pi$ and \Cref{sec:complexes-Hodge--Tate-sen-operator-for-regular-representation} is an isomorphism.
  Tensoring the resulting fiber sequence with $\mathcal{E}$ yields a fiber sequence
  \[
    \mathcal{E}\to \rho_{\pi,\ast}(\rho_{\pi}^\ast(\mathcal{E}))\xrightarrow{\alpha} \rho_{\pi,\ast}(\rho_{\pi}^\ast(\mathcal{E})).
  \]
  From here we apply the functor $R\Gamma(\mathrm{Spf}(\mathcal{O}_K)^{\rm HT},-)$ to get a fiber sequence
  \[
    R\Gamma(\mathrm{Spf}(\mathcal{O}_K)^{\rm HT},\mathcal{E})\to \beta_\pi^+(\mathcal{E})\xrightarrow{\widetilde{\alpha}} \beta^+_\pi(\mathcal{E})
  \]
  and one checks that $\widetilde{\alpha}=\Theta_{\pi,\mathcal{E}}$.
\end{proof}

\begin{remark}
  \label{sec:complexes-Hodge--Tate-characterization-of-bk-twists-over-r} Given \Cref{sec:complexes-Hodge--Tate-1-complexes-on-ht-locus-for-r} the same argument as in \cite[Corollary 3.5.14]{bhatt2022absolute} characterizes objects isomorphic to $\mathcal{O}_{\mathrm{Spf}(\mathcal{O}_K)^{\rm HT}}\{n\}\in \mathcal{D}(\mathrm{Spf}(\mathcal{O}_K)^{\rm HT})$ as those $\mathcal{E}\in \mathcal{D}(\mathrm{Spf}(\mathcal{O}_K)^{\rm HT})$ such that $\beta^+_\pi(\mathcal{E})\cong \mathcal{O}_K$ and $\Theta_{\pi,\mathcal{E}}=e\cdot n$.
\end{remark}

Now we prove the analog of \cite[Proposition 3.5.15]{bhatt2022absolute}.

\begin{proposition}
  \label{sec:complexes-Hodge--Tate-complexes-on-ht-divisor-generated-by-bk-twists}
  The category $\mathcal{D}(\mathrm{Spf}(\mathcal{O}_K)^{\rm HT})$ is generated under colimits by the Breuil--Kisin twists $\mathcal{O}_{\mathrm{Spf}(\mathcal{O}_K)^{\rm HT}}\{n\}$, $n\geq 0$.
\end{proposition}
\begin{proof}
  Following the argument of \cite[Proposition 3.5.15]{bhatt2022absolute} the claim reduces to showing that the regular representation $\mathcal{O}_{G_\pi}$ of $G_\pi$ lies in the category spanned  by the (objects corresponding to the) $\mathcal{O}_{\mathrm{Spf}(\mathcal{O}_K)^{\rm HT}}\{n\}$, $n\geq 0$. 
For this we can use the filtration of $\mathcal{O}_{G_\pi}=\widehat{\bigoplus\limits_{i\geq 0}} \mathcal{O}_K\cdot \frac{a^i}{i!}$ given by $\mathcal{F}_{\leq n}:=\bigoplus\limits_{i=0}^n \mathcal{O}_K\cdot \frac{a^i}{i!}$, the formula
  \[
    (1+ea)\frac{\partial }{\partial a}\frac{a^{n}}{n!}\equiv en \frac{a^n}{n!} \textrm{ mod }\mathcal{F}_{\leq {n-1}}
  \] and \Cref{sec:complexes-Hodge--Tate-characterization-of-bk-twists-over-r} to conclude the proposition as in \cite[Proposition 3.5.15]{bhatt2022absolute}.
\end{proof}

\begin{remark}
\label{ramified-case-only-structure-sheaf-needed}
In fact, when $K$ is ramified, one only needs the structure sheaf $\mathcal{O}_{\mathrm{Spf}(\mathcal{O}_K)^{\rm HT}}$ to generate the category under colimits: see \cite[Remark 9.7]{bhatt2022prismatization}. 
\end{remark}

We can now prove \Cref{sec:complexes-Hodge--Tate-1-complexes-on-ht-locus-for-r}.

\begin{proof}[{Proof of \Cref{sec:complexes-Hodge--Tate-1-complexes-on-ht-locus-for-r}}]
  Given \Cref{sec:complexes-Hodge--Tate-complexes-on-ht-divisor-generated-by-bk-twists}, well-definedness and fully faithfulness follow as in \cite[Theorem 3.5.8]{bhatt2022absolute}. 
The condition that $\Theta_e^p-e^{p-1}\Theta_e$ is locally nilpotent on the cohomology modulo $\pi$ appears as
  \[
    (en)^p-e^{p-1}(en)=e^p(n^p-n)
  \]
  is zero modulo $p$ for any $n\in \Z$.
  To show essential surjectivity let $M\in \mathcal{D}(\mathcal{O}_K[\Theta_\pi])$ be an object satisfying the two conditions in \Cref{sec:complexes-Hodge--Tate-1-complexes-on-ht-locus-for-r}. 
Let $\mathcal{O}_K\{n\}\in \mathcal{D}(\mathcal{O}_K[\Theta_\pi])$ be the image of $\mathcal{O}_{\mathrm{Spf}(\mathcal{O}_K)^{\rm HT}}\{n\}$. 
If $M$ is non-zero, we have to cook up a non-zero morphism $\mathcal{O}_K\{n\}[m]\to M$ for some $n,m\in \Z, n\geq 0$. 
As
  \[
    R\Gamma(\mathrm{Spf}(\mathcal{O}_K)^{\rm HT}, (\mathcal{O}_K\{-n\}[-m]\otimes_{\mathcal{O}_K} M)\otimes^L_{\mathcal{O}_K} k)\cong R\Gamma(\mathrm{Spf}(\mathcal{O}_K)^{\rm HT}, \mathcal{O}_K\{-n\}[-m]\otimes_{\mathcal{O}_K} M)\otimes^L_{\mathcal{O}_K} k 
  \]
  and $R\Gamma(\mathrm{Spf}(\mathcal{O}_K)^{\rm HT},\mathcal{O}_K\{-n\}[-m]\otimes M)$ is $p$-complete (e.g., by \Cref{sec:complexes-Hodge--Tate-cohomology-on-Cartier--Witt-stack}) it suffices to replace $M$ by $M\otimes^L_{\mathcal{O}_K} k$ (this uses $p$-completeness of $M$ to ensure that $M\otimes^L_{\mathcal{O}_K} k\neq 0$ if $M\neq 0$).
  Then $\Theta_\pi^p-e^{p-1}\Theta_\pi=\prod\limits_{i=0}^{p-1}(\Theta_\pi-e\cdot i)$ acts locally nilpotently on $H^\ast (M)$. 
Now we can finish the argument as in the proof of \cite[Theorem 3.5.8.]{bhatt2022absolute}. 
\end{proof}

\begin{remark}
	\label{sec:complexes-Hodge--Tate-cartier-dual}
	The Cartier duals of $\mathbb{G}_m^\sharp$ and $\mathbb{G}_a^\sharp$ are known, cf.\ \cite[Remark 3.5.17]{bhatt2022absolute}, \cite[Appendix B]{drinfeld20211}. 
	For $G_\pi$ one can uniformly describe the Cartier dual as the formal group scheme
	\[
	\Spf(\varprojlim\limits_{n} \mathcal{O}_K[u]/\prod\limits_{i=0}^n (u-e\cdot i)).
	\]
\end{remark}
\begin{remark}
\label{sec:complexes-Hodge--Tate-cartier-dual-bis}
\Cref{sec:complexes-Hodge--Tate-1-complexes-on-ht-locus-for-r} could also be deduced from \Cref{sec:complexes-Hodge--Tate-cartier-dual} and Cartier duality (switching to modules over the Cartier dual of $G_\pi$) and the observation we already made that \[u^p-e^{p-1}u = \prod_{i=0}^{p-1} (u-e\cdot i)\]
in $k[u]$. 
\end{remark}

For perfect complexes we get the following version of \Cref{sec:complexes-Hodge--Tate-1-complexes-on-ht-locus-for-r}.

\begin{lemma}
  \label{sec:from-complexes-hodge-perfect-complexes-for-ok-ht}
  The functor $\beta^+_\pi$ from \Cref{sec:complexes-Hodge--Tate-1-complexes-on-ht-locus-for-r} restricts to a fully faithful functor
  \[
    \beta^+_\pi\colon \mathcal{P}erf(\Spf(\O_K)^\HT)\to \mathcal{P}erf(\O_K[\Theta_\pi])
  \]
  whose essential image consists of $\pi$-adically complete perfect complexes $M$ over $\O_K[\Theta_\pi]$ for which $\Theta^p_\pi-e^{p-1}\Theta_\pi$ is nilpotent on $H^\ast(k\otimes^L_{\O_K}M)$. 
\end{lemma}
The perfect $\O_K[\Theta_\pi]$-module $M=\O_K[\Theta_\pi]/(\pi\Theta_\pi-1)$-module is zero modulo $\pi$. Hence, the $\pi$-completeness assumption on $M$ is necessary.
\begin{proof}
  By regularity of $\O_K[\Theta_\pi]$, a complex of $\O_K[\Theta_\pi]$-modules is perfect if its underlying complex of $\O_K$-modules is perfect. This shows that $\beta^+_\pi$ sends $\mathcal{P}erf(\Spf(\O_K)^\HT)$ to $\mathcal{P}erf(\O_K[\Theta_\pi])$. Then fully faithfulness follows from \Cref{sec:complexes-Hodge--Tate-1-complexes-on-ht-locus-for-r}. 
  
  Conversely, let $M\in \mathcal{P}erf(\O_K[\Theta_\pi])$ be $\pi$-adically complete perfect complex such that $\Theta^p_\pi-e^{p-1}\Theta_\pi$ is nilpotent on $H^\ast(k\otimes^L_{\O_K}M)$. We need to see that $M$ is perfect as a complex of $\O_K$-modules. It suffices to check this for $M\otimes^L_{\O_K} \O_K/\pi^n$ for all $n\geq 0$. Indeed, as $\O_K$ is $\pi$-adically complete this implies that $M$ is perfect as a complex of $\O_K$-modules. As canonical truncations of $M$ are again perfect $\O_K[\Theta_\pi]$-modules, this reduces to the case that $M$ is concentrated in a single degree. Using d\'evissage, we reduce to the case that $\pi M=0$. But then $M$ is a finitely generated $\O_K[\Theta_\pi]/(\Theta^p_\pi-e^{p-1}\Theta_\pi)^i$-module for some $i\geq 0$ and hence perfect as an $\O_K$-module.  
\end{proof}

Let $S$ be a $p$-complete $\mathcal{O}_K$-algebra and $g=(s,b)\in G_\pi(S)$ an $S$-valued point of $G_\pi$. 
As $G_\pi(S)$ is commutative, the multiplication by $g$ will induce an automorphism
\[
  \gamma_{g,M}\colon M\widehat{\otimes}_{\mathcal{O}_K}S\to M\widehat{\otimes}_{\mathcal{O}_K}S
\]
on any pair $(M,\Theta_\pi)\in \mathcal D(BG_\pi)$. 
The next lemma makes this action explicit, and also sheds more light on the condition on $\Theta_\pi$ in \Cref{t:beta}.
The formula also occurs in \cite[Theorem 1.3]{min2021hodge}.

\begin{lemma}
  \label{sec:complexes-Hodge--Tate-explicit-action} In the above notation, we have the equality
  \[
    \gamma_{g,M}=(1+e\cdot b)^{{\Theta_\pi}/e}:=\sum\limits_{n=0}^\infty \frac{b^n}{n!}\prod\limits_{i=0}^{n-1}(\Theta_\pi-e\cdot i)
  \]
  as endomorphisms of each cohomology module of $M\widehat{\otimes}_{\mathcal{O}_K}S$. 
\end{lemma}

We recall that there exist the formal equalities
\begin{alignat*}{2}
    \exp(x\cdot \log(1+y))
  &  = && (1+y)^x\\
  &  = && \sum\limits_{n\geq 0} \binom{x}{n} y^n\\
   & = && \sum\limits_{n\geq 0}\frac{y^n}{n!}\prod\limits_{i=0}^{n-1}(x-i)
  \end{alignat*}
of power series in $\Q[[x,y]]$.

\begin{proof}
  For simplicity of notation, we assume $\mathcal{O}_K=S$. 
The general case is similar.
  Using a reduction to $M$ concentrated in degree $0$ and the natural morphism $M\to \rho_{\pi,\ast}(\rho^{\ast}_\pi M)$, it suffices to check the statement in the case that $M=\calO_{G_\pi}$ with $\Theta_\pi=(1+e\cdot a)\frac{\partial}{\partial a}$.
  Let $f(a)\in \calO_{G_\pi}=\widehat{\bigoplus\limits_{n\geq 0}} \mathcal{O}_K\frac{a^n}{n!}$.
  Then
  \[
    \gamma_{g,\calO_{G_\pi}}(f)(a)=f(a+b+e\cdot ab).
  \]
  We write $t:=1+e\cdot a$ and $h(t)=f(\frac{t-1}{e})$. 
Then $\gamma_{g,\calO_{G_\pi}}(f)(a)=h(s\cdot t)$ with $s=1+e\cdot b$. 
From the proof of \cite[Proposition 3.7.1]{bhatt2022absolute} we get
  \begin{alignat*}{2}
       h(s\cdot t)&= \sum\limits_{n\geq 0}\frac{\log(s)^n}{n!}(t\frac{\partial}{\partial t})^n h(t) \\
      &=  \exp(\frac{\Theta_\pi}{e}\log(s)) h(t)
\end{alignat*}
 from which we deduce the formal equality of power series in $a$
  \begin{equation}
    \label{eq:1}
       f(a+b+e\cdot ab)=(1+e\cdot b)^{\Theta_\pi/e}f(a).
  \end{equation} 
  From the formal equality
  \begin{equation}
    \label{eq:2}
       (1+y)^x=\sum\limits_{n\geq 0}\frac{y^n}{n!}\prod\limits_{i=0}^{n-1}(x-i)
  \end{equation} 
  mentioned before we conclude that (\Cref{eq:1}) converges in $\mathcal{O}_{G_\pi}$ by the convergence condition on $\Theta_\pi$. 
Namely, setting $y=e\cdot b$ and $x=\Theta_\pi/e$ we get
 \begin{alignat*}{2}
      & (1+e\cdot b)^{\Theta_\pi/e}
       &=& \sum\limits_{n\geq 0}\frac{e^nb^n}{n!}\prod\limits_{i=0}^{n-1}(\Theta_\pi/e-i) \\
      &&= &\sum\limits_{n\geq 0}\frac{b^n}{n!}\prod\limits_{i=0}^{n-1}(\Theta_\pi-i\cdot e) \end{alignat*}
  and $\prod\limits_{i=0}^{n-1}(\Theta_\pi-i\cdot e)$ converges to $0$ if $n\to \infty$.
\end{proof}
\subsection{An analytic variant}
\label{sec:an-analytic-variant-analytic-variant}
Finally, when we restrict attention to perfect complexes, we deduce an ``analytic'' variant of the functor $\beta_{\pi}^+$, which will be more closely related to $v$-vector bundles: Namely, by formally inverting $p$ on the source, $\beta^+_{\pi}$ defines a functor
\[
\beta_{\pi}:\mathcal{P}erf(\mathrm{Spf}(\mathcal{O}_K)^{\rm HT})\tf\to \mathcal{P}erf(K[\Theta_\pi]).\]
\begin{corollary}
  \label{sec:an-analytic-variant-corollary-analytic-variant}
 The functor $
\beta_{\pi}$
is fully faithful. Its essential image consists of complexes $M\in \mathcal{P}erf(K[\Theta_\pi])$ such that $H^\ast(M)$ is finite dimensional over $K$ and the action of $\Theta^p_\pi-e^{p-1}\Theta_\pi$ on  $H^\ast(M)$ is topologically nilpotent.
\end{corollary}
Here, the cohomology $H^\ast(M)$ has its canonical topology as a finite dimensional $K$-vector space.
\begin{proof}
  By \Cref{sec:from-complexes-hodge-perfect-complexes-for-ok-ht} we can identify $\mathcal{P}erf(\Spf(\O_K)^\HT)$ with the full subcategory of $\mathcal{P}erf(\O_K[\Theta_\pi])$ given by $\pi$-complete objects $M$ such that $\Theta^p_\pi-e^{p-1}\Theta_\pi$ is nilpotent on $H^\ast(k\otimes^L_{\O_K}M)$.
  The natural functor $\mathcal{P}erf(\O_K[\Theta_\pi])\tf\to \mathcal{P}erf(K[\Theta_\pi])$ is fully faithful because perfect complexes are compact objects. This establishes fully faithfulness of $\beta_\pi$. 
  
  To see the description of the essential image, by induction on the amplitude and via considering cones, we can reduce to the case that $M$ is concentrated in degree $0$. Then $\Theta_\pi\colon M\to M$ is an endomorphism of a finite dimensional $K$-vector space whose characteristic polynomial has coefficients in $\O_K$. This implies that there exists a $\Theta_\pi$-stable $\O_K$-lattice $M^0$ in $M$. Indeed, if $L/K$ is a finite extension and $N^0\subseteq M\otimes_KL$ a $\Theta_\pi$-stable $\O_L$-lattice, then $M^0:=N^0\cap M$ is a $\Theta_\pi$-stable $\O_K$-lattice in $M$. Hence, for the existence of $M^0$ we may enlarge $K$. Then $\Theta_\pi$ can be assumed to have Jordan normal form, in which case the existence of $M^0$ is clear as all eigenvalues lie in $\O_K$ by integrality of the characteristic polynomial.
  Now any $\Theta_\pi$-stable $\O_K$-lattice $M^0$ in $M$ is $\pi$-adically complete and the $\Theta_\pi^p-e^{p-1}\Theta_\pi$-action on $M^0/\pi$ is nilpotent by the assumed topological nilpotence. By \Cref{sec:from-complexes-hodge-perfect-complexes-for-ok-ht} this implies that $M$ lies in the essential image as desired.
       \end{proof}

\section{Galois actions on Hodge--Tate stacks}
\label{sec:galo-acti-cart}

Let $C=\widehat{\overline{K}}$ be the completion of an algebraic closure of $K$. 
Let
\[
  f\colon \Spf(\mathcal{O}_C)\to \Spf(\mathcal{O}_K) 
\]
denote the natural morphism. 
As the natural map $\mathrm{Spf}(\mathcal{O}_C)^{\rm HT}\to \Spf(\mathcal{O}_C)$ is an isomorphism (\cite[Example 3.12]{bhatt2022prismatization}), the map $f$ lifts naturally to a map
\[
  \tilde{f}\colon \Spf(\mathcal{O}_C)\to \mathrm{Spf}(\mathcal{O}_K)^{\rm HT}
\]
over $\Spf(\mathcal{O}_K)$. 
Let $Z_\pi$ be the base change along $\tilde{f}$ of the $G_\pi$-torsor $\rho_\pi: \mathrm{Spf}(\mathcal{O}_K) \to \mathrm{Spf}(\mathcal{O}_K)^{\rm HT}$.

In this section, we would like to analyze explicitly the action of $\mathrm{Gal}(\overline{K}/K)$ on the ring of functions $\mathcal{O}(Z_\pi)$ under some additional choices. 
We do so in \Cref{sec:expl-funct-explicit-functoriality} in a slightly more general context. 
Using this description, we will relate in \Cref{sec:comp-colm-ring-comparison-to-b-sen} the ring $\mathcal{O}(Z_\pi)[1/p]$ to other more familiar period rings from $p$-adic Hodge theory (most notably the ring $B_{\rm Sen}$ previously introduced by Colmez) and compute its Galois cohomology in \Cref{sec:calc-galo-cohom}.

\subsection{Explicit functoriality }
\label{sec:expl-funct-explicit-functoriality}

Let $K, K^\prime$ be $p$-adic fields with rings of integers $R:=\mathcal{O}_K, R^\prime:=\mathcal{O}_{K^\prime}$ and (perfect) residue fields $k,k^\prime$.
We assume that $K\to C, K^\prime\to C^\prime$ are two extensions of non-archimedean fields with $C, C^\prime$ algebraically closed. 
Moreover, we let $\sigma\colon C\to C^\prime$ be a continuous homomorphism, which we require to induce a homomorphism $\tau\colon K\to K^\prime$.
Let
\[
  f\colon \Spf(\mathcal{O}_C)\to \Spf(R),\ f^\prime\colon \Spf(\mathcal{O}_{C^\prime})\to \Spf(R^\prime)
\]
denote the natural morphisms. 
As recalled above in the particular case $C=\widehat{\overline{K}}$, the natural map $\mathrm{Spf}(\mathcal{O}_C)^{\rm HT}\to \Spf(\mathcal{O}_C)$ is an isomorphism (\cite[Example 3.12]{bhatt2022prismatization}) and thus the map $f$ lifts naturally to a map
\[
  \tilde{f}\colon \Spf(\mathcal{O}_C)\to \mathrm{Spf}(R)^{\rm HT}
\]
over $\Spf(R)$. 
Similarly, $f^\prime$ lifts naturally to a map $\tilde{f^\prime}\colon \Spf(\mathcal{O}_{C^\prime})\to \mathrm{Spf}(R^\prime)^{\rm HT}$.
Explicitly, if $S$ is a $p$-adically complete $R$-algebra, $(A_\inf,J)$ the perfect prism associated to $\mathcal{O}_C$, then $\tilde{f}$ maps a morphism $g\colon \mathcal{O}_C\to S$ to the image of the $\mathcal{O}_C$-point
\[
  \left(J\otimes_{A_{\inf}} W(\mathcal{O}_C)\to W(\mathcal{O}_C), R\to \mathcal{O}_C\cong A_\inf/J\to \overline{W(\mathcal{O}_C)}\right)\in \mathrm{Spf}(R)^{\rm HT}(\mathcal{O}_C)
\]
along the map $\mathrm{Spf}(R)^{\rm HT}(\mathcal{O}_C)\to \mathrm{Spf}(R)^{\rm HT}(S)$ induced by $g$. 
Similarly the map $\tilde{f^\prime}$ can be described using the perfect prism $(A_\inf^\prime, J^\prime)$ associated with $\mathcal{O}_{C^\prime}$.

From the naturality of the Hodge--Tate stack we deduce that there exists a natural $2$-commutative diagram
\[
  \xymatrix{
    \Spf(\mathcal{O}_{C^\prime})\ar[r]^\sigma\ar[d]^{\tilde{f^\prime}} & \Spf(\mathcal{O}_C)\ar[d]^{\tilde{f}} \\
    \mathrm{Spf}(R^\prime)^{\rm HT} \ar[r]^\tau & \mathrm{Spf}(R)^{\rm HT},
  }
\]
or more precisely a natural isomorphism $\iota_\sigma\colon \tilde{f}\circ \sigma\cong \tau\circ\tilde{f}^\prime$ between two points in the groupoid $\mathrm{Spf}(R)^{\rm HT}(\mathcal{O}_{C^\prime})\cong \mathrm{Mor}_R(\Spf(\mathcal{O}_{C^\prime}), \mathrm{Spf}(R)^{\rm HT})$.

Set $\tilde{\pi}:=\tau(\pi)\in R^\prime$ and assume that this element is a uniformizer in $R^\prime$. 
We get two prisms 

\[
  (A_\pi, I_\pi) = (W(k)[[u]], (E_\pi(u))) \]
  and 
  \[    (A^\prime_{\tilde{\pi}}, I_{\tilde{\pi}}^\prime) = (W(k^\prime)[[u^\prime]], (E_{\tilde{\pi}}(u')))
\]
 lifting $R$ and $R^\prime$ respectively via the morphisms
\[
  A_\pi=W(k)[[u]]\to R, u \mapsto \pi, ~~ A^\prime_{\tilde{\pi}}=W(k^\prime)[[u^\prime]]\to R^\prime,\ u^\prime\mapsto \tilde{\pi}.
\]
As explained in \Cref{sec:recollection-ht-cw-stack}, these choices give rise to two maps
\[
  \rho_\pi : \mathrm{Spf}(R) \to \mathrm{Spf}(R)^{\rm HT}, ~~ \rho_{\tilde{\pi}}^\prime\colon \Spf(R^\prime)\to \mathrm{Spf}(R^\prime)^{\rm HT}
\]
with group sheaves of automorphisms
\[
  G_\pi = \{(t,a)\in \mathbb{G}_m^\sharp\ltimes \mathbb{G}_a^\sharp\ |\ t=1+e \cdot a\}, ~~ G_{\tilde{\pi}}=\{(t,a)\in \mathbb{G}_m^\sharp\ltimes \mathbb{G}_a^\sharp\ |\ t=1+\tilde{e}\cdot a\},
\]
where
\[
 e:=E^\prime_\pi(\pi), ~~  \tilde{e}:=E^\prime_{\tilde{\pi}}(\tilde{\pi}).
\]
Let $\tau_0\colon W(k)\to W(k^\prime)$ be the homomorphism induced by $\tau$. 
Then $\tau_0$ extends to a homomorphism
\[
  \tau_A\colon A_\pi\to A_{\tilde{\pi}}^\prime
\]
by sending $u$ to $u^\prime$. 
This yields a homomorphism $(A_\pi, I_\pi)\to (A_{\tilde{\pi}}^\prime, I_{\tilde{\pi}}^\prime)$ of prisms, which reduces to the homomorphism $\tau\colon R\cong A_\pi/I_\pi\to A_{\tilde{\pi}}^\prime/I^\prime_{\tilde{\pi}}\cong R^\prime$. 
By naturality of \cite[Construction 3.10]{bhatt2022prismatization} we obtain a $2$-commutative diagram
\[
  \xymatrix{
    \Spf(R^\prime)\ar[r]^{\tau}\ar[d]^{\rho_{\tilde{\pi}}^\prime} & \Spf(R) \ar[d]^{\rho_\pi} \\
    \mathrm{Spf}(R^\prime)^{\rm HT} \ar[r]^\tau & \mathrm{Spf}(R)^{\rm HT}.
  }
\]
To describe the implicit isomorphism in $\mathrm{Spf}(R)^{\rm HT}(R^\prime)$, note that the diagram
\[
  \xymatrix{
    A_\pi \ar[r]\ar[d]^{\tau_A} & W(R)\ar[d]^{W(\tau)} \\
    A^\prime_{\tilde{\pi}} \ar[r]& W(R^\prime)
  }
\]
commutes and that the natural map $I_\pi\otimes_{A_\pi}A^\prime_{\tilde{\pi}}\to I^\prime_{\tilde{\pi}}$ is an isomorphism. 
From here one sees that the implicit isomorphism is induced by the isomorphism
\[
  I_\pi\otimes_{A_\pi} W(R)\otimes_{W(R)}W(R^\prime)\to I_{\tilde{\pi}}\otimes_{A_{\tilde{\pi}}} W(R^\prime),\ i\otimes x\otimes y\mapsto \tau_A(i)\otimes W(\tau)(x)y.
\]
As in the beginning of \Cref{sec:galo-acti-cart}, we define $Z_\pi\to \Spf(\mathcal{O}_C)$ by the fiber product diagram
\[
  \xymatrix{
    Z_\pi\ar[r]\ar[d] & \Spf(\mathcal{O}_C)\ar[d]^{\tilde{f}} \\
    \Spf(R) \ar[r]^{\rho_\pi} & \mathrm{Spf}(R)^{\rm HT}.
  }
\]
By faithfully flat descent for $G_\pi$-torsors
\[
  Z_\pi\cong \Spf(A_\en)
\]
for some $\mathcal{O}_C$-algebra $A_\en$ with $G_\pi$-action.
Similarly, we can define $Z^\prime_{\tilde{\pi}}=\Spf({A^{\prime}_{\en}})$.
From the $2$-commutative diagram
\[
  \begin{tikzcd}
     \Spf(R^\prime) \arrow[r,"{\rho^\prime_{\tilde{\pi}}}"]\arrow[d,"\tau"] & \mathrm{Spf}(R^\prime)^{\rm HT} \arrow[d] & \Spf(\mathcal{O}_{C^\prime})\arrow[l,"{\tilde{f^\prime}}"']\arrow[d, "{\sigma}"]\\
     \Spf(R) \arrow[r,"{\rho_\pi}"] & \mathrm{Spf}(R)^{\rm HT} & \Spf(\mathcal{O}_C)\arrow[l,"\tilde{f}"'] 
   \end{tikzcd}
\]
 we deduce from the universal property of the fiber product a natural map
 \[
   \sigma_Z\colon Z_{\tilde{\pi}}^\prime\to Z_{\pi} 
 \]
Our aim is to make $\sigma_Z$ more explicit. 
For this we first have to construct a suitable trivialization of the $G_\pi$-torsor $Z_\pi\to \Spf(\mathcal{O}_C)$ and the $G_{\tilde{\pi}}$-torsor $Z^\prime_{\tilde{\pi}}\to \Spf(\mathcal{O}_{C^\prime})$.

If $\pi^\flat=(\pi, \pi^{1/p},\ldots)\in \mathcal{O}_C^\flat$ is a compatible system of $p$-power roots of $\pi$, then the embedding
\[
  \iota_{\pi^\flat}\colon A_\pi\to A_\inf,\ u\mapsto [\pi^\flat]
\]
extends to a morphism of prisms $(A_\pi,I_\pi)\to (A_\inf,J)$ and yields an isomorphism $\gamma_{\pi^\flat}$ in $\mathrm{Spf}(R)^{\rm HT}(\mathcal{O}_C)$ between the points corresponding to $\tilde{f}$ and the composition $\Spf(\mathcal{O}_C)\to \Spf(R)\xrightarrow{\rho_\pi} \mathrm{Spf}(R)^{\rm HT}$. 
Concretely, $\gamma_{\pi^\flat}$ is the isomorphism of objects in $\mathrm{Spf}(R)^{\rm HT}(\O_C)$
\[
\begin{tikzcd}[column sep = {0.3cm},row sep =0.15cm]
	\Big(I_\pi\otimes_{A_\pi} W(\mathcal{O}_C) \arrow[dd] \arrow[r] & W(\mathcal{O}_C) \arrow[dd, equal] & {,} & R\arrow[r,equal]\arrow[dd]& A_\pi/I_\pi  \arrow[r]    & \overline{W(\mathcal{O}_C)}\Big) \arrow[dd, equal] \\
	&                                                  &     &                                            &                                             \\
	\Big(J\otimes_{A_\inf}W(\mathcal{O}_C) \arrow[r]                & W(\mathcal{O}_C)                                 & {,} & R\arrow[r]& \mathcal{O}_C\cong A_\inf/J \arrow[r] & \overline{W(\mathcal{O}_C)}\Big)           
\end{tikzcd}\]
induced by the $W(\mathcal{O}_C)$-linear isomorphism
\[
  I_\pi\otimes_{A_\pi} W(\mathcal{O}_C)\cong J\otimes_{A_\inf} W(\mathcal{O}_C),\ i\otimes x\mapsto \iota_{\pi^\flat}(i)\otimes x.
\]
Similarly, a choice $\tilde{\pi}^\flat=(\tilde{\pi},\tilde{\pi}^{1/p},\ldots)\in \mathcal{O}_{C'^\flat}$ of $p$-power compatible $p$-power roots of $\tilde{\pi}$ in $\mathcal{O}_{C^\prime}$ yields an isomorphism between $\tilde{f^\prime}$ and the composition
\[
  \Spf(\mathcal{O}_{C^\prime})\to \Spf(R^\prime)\xrightarrow{\rho^\prime_{\tilde{\pi}}} \mathrm{Spf}(R^\prime)^{\rm HT}.
\]
We note that a possible such choice is given by $\sigma(\tilde{\pi}^\flat)$, but we explicitly want to allow greater freedom for this choice to later compute Galois actions.

From here we see that $\gamma_{\pi^\flat}, \gamma_{\tilde{\pi}^\flat}$ induce isomorphisms
\[
  \gamma_{\pi^\flat}\colon G_{\pi,\mathcal{O}_C}:= G_{\pi}\times_{\Spf(R)}\Spf(\mathcal{O}_C)\xrightarrow{\simeq} Z_\pi
\]
and
\[
  \gamma_{\tilde{\pi}^\flat}\colon G_{\tilde{\pi},\mathcal{O}_{C^\prime}}:=G_{\tilde{\pi}}\times_{\Spf(R^\prime)}\Spf(\mathcal{O}_{C^\prime})\xrightarrow{\simeq} Z^\prime_{\tilde{\pi}}. 
\]
 We can conjugate the map $\sigma_Z\colon Z^\prime_{\tilde{\pi}}\to Z_\pi$ to get the composition 
\[
  G_{\tilde{\pi},\mathcal{O}_{C^\prime}}\xisomarrow{\gamma_{\tilde{\pi}^\flat}} Z^\prime_{\tilde{\pi}}\xrightarrow{\sigma_Z} Z_\pi \xisomarrow{\gamma_{\pi^\flat}^{-1}} G_{\pi,\mathcal{O}_C},
\]
which lies over the morphism $\sigma\colon \Spf(\mathcal{O}_{C^\prime})\to \Spf(\mathcal{O}_C)$. 
Base changing along $\sigma$ yields a map
\[
  \sigma_{G}\colon G_{\tilde{\pi},\mathcal{O}_{C^\prime}}\to \sigma^\ast G_{\pi,\mathcal{O}_C}:=G_{\pi}\times_{\Spf(R)}\Spf(\mathcal{O}_{C^\prime})
\]
(we write $\sigma_G^{\pi^\flat, \tilde{\pi}^\flat}$ if we want to stress its dependence on the choices of $\pi^\flat$ and $\tilde{\pi}^\flat$).

We now choose a compatible system
\[
  \varepsilon=(1,\zeta_p,\ldots)\in \mathcal{O}_{C^\prime}^\flat
\]
of $p$-power primitive roots of unity.
Then we can write
\[
  \sigma([\pi^\flat])=[\varepsilon]^{c(\sigma)}[\tilde{\pi}^\flat]\in A^\prime_{\inf}
\]
for a unique element $c(\sigma)\in \Z_p$. 
Indeed, this follows as $\sigma(\pi)=\tau(\pi)=\tilde{\pi}$.
Let us note that
\[
  \tau_A(E_\pi(u))=E_{\tilde{\pi}}(u^\prime).
\]
We furthermore introduce the elements
\[
  \mu:=[\varepsilon]-1\in A^\prime_\inf
\]
and
\[
  z:=\tilde{\pi}\cdot \theta^\prime(\frac{\mu}{E_{\tilde{\pi}}([\tilde{\pi}^\flat])})\in \mathcal{O}_{C^\prime}
\]
(here $\theta^\prime\colon A^\prime_\inf\to \mathcal{O}_{C'}$ is Fontaine's map).
Let us note that if $\xi:=\frac{\mu}{\varphi^{-1}(\mu)}=1+[\varepsilon^{\frac{1}{p}}]+\ldots + [\varepsilon^{\frac{p-1}{p}}]$, then $\frac{\xi}{E_{\tilde{\pi}}([\tilde{\pi}^\flat])} \in A_{\rm inf}^\prime$ is a unit and hence also $\theta^\prime(\frac{\xi}{E_{\tilde{\pi}}([\tilde{\pi}^\flat])}) \in \mathcal{O}_{C^\prime}$ is a unit. 
Therefore 
\[
  z = \tilde{\pi} \cdot \theta^\prime(\frac{\xi}{E_{\tilde{\pi}}([\tilde{\pi}^\flat])}) \theta^\prime(\varphi^{-1}(\mu))  \in \tilde{\pi}\cdot (\zeta_p-1)\mathcal{O}_C.
\]
In particular, $z\in \mathbb{G}_a^\sharp(\mathcal{O}_{C^\prime})$.

The main result of this subsection is the following statement.
\begin{proposition}
  \label{sec:galo-acti-cart-1-description-of-sigma-z}
  The map
  \[
    \sigma_G\colon G_{\tilde{\pi},\mathcal{O}_{C^\prime}}\to \sigma^\ast G_{\pi,\mathcal{O}_C}
  \]
  is given by the map
  \[
    (t,a)\mapsto ((1+\tilde{e}\cdot z\cdot c(\sigma))^{-1}\cdot t, \frac{1}{1+\tilde{e}z c(\sigma)}\frac{\tilde{e}}{e}(a-z\cdot c(\sigma)). 
  \]
\end{proposition}
\begin{proof}
  By flatness of both formal schemes over $\Spf(\mathcal{O}_{C^\prime})$ it suffices to identify $\sigma_G$ on $\mathcal{O}_{C^\prime}$-algebras $S$ which are $p$-torsion free. 
In particular, a point $(s,b)\in \sigma^\ast G_{\pi,\mathcal{O}_C}(S)\subseteq \mathbb{G}_m^\sharp(S)\ltimes \mathbb{G}_a^\sharp(S)$ is already determined by $s$ as the equation $s=1+e\cdot b\in S$ holds.
  Now, assume that $(t,a)\in G_{\tilde{\pi},\mathcal{O}_C}(S)$ is a given point, and let $g\in W^\times[F](S)\cong \mathbb{G}_m^\sharp(S)$ be the element corresponding to $t$ (in particular, $g=(t,\ldots,)$ in Witt coordinates). 
Now $\sigma_G(t,a)$ is by definition the element in $\sigma^\ast G_{\pi,\mathcal{O}_C}$ determined by the isomorphism of the two outer compositions in the $2$-commutative diagram
  
\[\begin{tikzcd}
	&& {\Spf(S)} \\
	\\
	{\Spf(R^\prime)} && {\mathrm{Spf}(R^\prime)^{\rm HT}} & {\Spf(R^\prime)} & {\Spf(\mathcal{O}_{C^\prime})} \\
	{\Spf(R^\prime)} && {\mathrm{Spf}(R^\prime)^{\rm HT}} && {\Spf(\mathcal{O}_{C^\prime})} \\
	{\Spf(R)} && {\mathrm{Spf}(R)^{\rm HT}} && {\Spf(\mathcal{O}_C)} \\
	{\Spf(R)} && {\mathrm{Spf}(R)^{\rm HT}} & {\Spf(R)} & {\Spf(\mathcal{O}_C)}
	\arrow["{\rho_\pi}", from=6-1, to=6-3]
	\arrow["{\rho_\pi}"', from=6-4, to=6-3]
	\arrow[from=6-5, to=6-4]
	\arrow[from=5-5, to=6-5,equal]
	\arrow[from=5-3, to=6-3,equal]
	\arrow[from=5-1, to=6-1,equal]
	\arrow["{\rho_\pi}", from=5-1, to=5-3]
	\arrow["{\tilde{f}}"', from=5-5, to=5-3]
	\arrow["\sigma", from=4-5, to=5-5]
	\arrow["\tau", from=4-3, to=5-3]
	\arrow["\tau", from=4-1, to=5-1]
	\arrow["{\rho_{\tilde{\pi}}^\prime}", from=4-1, to=4-3]
	\arrow["{\tilde{f^\prime}}"', from=4-5, to=4-3]
	\arrow[from=3-5, to=4-5,equal]
	\arrow[from=3-5, to=3-4]
	\arrow["{\rho_{\tilde{\pi}}^\prime}"', from=3-4, to=3-3]
	\arrow["{\rho_{\tilde{\pi}}^\prime}", from=3-1, to=3-3]
	\arrow[from=3-1, to=4-1, equal]
	\arrow[from=1-3, to=3-1]
	\arrow[from=1-3, to=3-5]
	\arrow[from=3-3, to=4-3, equal]
      \end{tikzcd}\]
    where the isomorphism of the top roof is defined by $g$, the isomorphism for the upper right square is given by $\gamma_{\tilde{\pi}^\flat}$ and the one for the lower right by $\gamma_{\pi^\flat}$. 
For the other squares the implicit isomorphisms are obtained by functoriality of the Hodge--Tate stack or of \cite[Construction 3.10]{bhatt2022prismatization}. 
Write $\sigma_G(t,a)=(s,b)$ and let $h\in W^\times[F](S)\cong \mathbb{G}_m^\sharp(S)$ be the point determined by $s$. Let us now for simplicity set $S=\O_{C'}$. Then
the commutativity of the outer diagram implies that the two compositions
    
     \[\begin{array}{ c l }
         &I_\pi\otimes_{A_\pi} W(\mathcal{O}_C)\otimes_{W(\mathcal{O}_C)} W(\mathcal{O}_{C^\prime}) \\
        \xrightarrow{\tau_A\otimes \mathrm{Id}_{W(\mathcal{O}_{C^\prime})}} & I_{\tilde{\pi}}^\prime\otimes_{A^\prime_{\tilde{\pi}}} W(\mathcal{O}_{C^\prime}) \\
        \xrightarrow{\cdot g} & I_{\tilde{\pi}}^\prime\otimes_{A^\prime_{\tilde{\pi}}} W(\mathcal{O}_{C^\prime}) \\
        \xrightarrow{\gamma_{\tilde{\pi}^\flat}} & J^\prime\otimes_{A^\prime_\inf} W(\mathcal{O}_{C^\prime}) \\
      \end{array}\]
    
    and
    \[
      \begin{array}{ c l }
        & I_\pi\otimes_{A_\pi} W(\mathcal{O}_C)\otimes_{W(\mathcal{O}_C)} W(\mathcal{O}_{C^\prime}) \\
        \xrightarrow{\cdot h} & I_\pi\otimes_{A_\pi} W(\mathcal{O}_C)\otimes_{W(\mathcal{O}_C)} W(\mathcal{O}_{C^\prime})\\
        \xrightarrow{\gamma_{\pi^\flat}\otimes \mathrm{Id}_{W(\mathcal{O}_{C^\prime})}} & J\otimes_{A_\inf} W(\mathcal{O}_C)\otimes_{W(\mathcal{O}_C)} W(\mathcal{O}_{C^\prime}) \\
        \xrightarrow{\sigma}& J^\prime\otimes_{A^\prime_{\inf}} W(\mathcal{O}_{C^\prime})
      \end{array}
    \]
    agree.
    Evaluating both morphisms on the element $E_{\pi}(u)\otimes 1\otimes 1$ yields the equality (in $J^\prime\otimes_{A^\prime_{\inf}} W(\mathcal{O}_{C^\prime})$)
    \[
      E_{\tilde{\pi}}([\tilde{\pi}^\flat])\otimes g=\sigma(E_\pi([\pi^\flat]))\otimes h.
    \]
    Now, $\sigma(E_\pi([\pi^\flat]))=E_{\tilde{\pi}}(\sigma([\pi^\flat]))$, which implies that
    \[
      h=\frac{E_{\tilde{\pi}}([\tilde{\pi}^\flat])}{E_{\tilde{\pi}}(\sigma([\pi^\flat]))}\cdot g,
    \]
    where $\frac{E_{\tilde{\pi}}([\tilde{\pi}^\flat])}{E_{\tilde{\pi}}(\sigma([\pi^\flat]))}\in A^\prime_\inf$ is a unit and the multiplication with it refers to the $A^\prime_\inf$-algebra structure of $W(S)$. 
By \Cref{sec:galo-acti-cart-2-description-of-quotient-of-eisenstein-polynomials-of-uniformizers} below, we can conclude that
    \[
      (1+\tilde{e}\cdot z \cdot c(\sigma))^{-1}\cdot t=s.
    \]
    From this we can conclude by a small calculation that
    \[
      b=\frac{1}{1+\tilde{e}z c(\sigma)}\frac{\tilde{e}}{e}(a-z\cdot c(\sigma))
    \]
    as desired.
\end{proof}

\begin{lemma}
  \label{sec:galo-acti-cart-2-description-of-quotient-of-eisenstein-polynomials-of-uniformizers}
  Under the map $\theta^\prime\colon A^\prime_{\inf}\to \mathcal{O}_{C^\prime}$ the element
  \[
    \frac{E_{\tilde{\pi}}(\sigma([\pi^\flat]))}{E_{\tilde{\pi}}([\tilde{\pi}^\flat])}\in A^\prime_{\inf}
  \]
  is mapped to the element
  \[
    1+\tilde{e} \cdot z\cdot c(\sigma)\in \mathcal{O}_C.
  \]
\end{lemma}
\begin{proof} 
  We write
  \[
    \sigma([\pi^\flat])=[\varepsilon]^{c(\sigma)}[\tilde{\pi}^\flat]
  \]
  and
  \[
    \mu:=[\varepsilon]-1.
  \]
  Recall that $\ker(A^\prime_\inf\to W(\mathcal{O}_{C^\prime}))=\mu\cdot A^\prime_{\inf}$, cf.\ \cite[Lemma 3.23]{Bhatt2018}.
  Now
  \[
    \begin{array}{r l}
     E_{\tilde{\pi}}(\sigma([\pi^\flat]))  = & E_{\tilde{\pi}}([\tilde{\pi}^\flat](1+\mu)^{c(\sigma)}) \\
      = & E_{\tilde{\pi}}([\tilde{\pi}^\flat](1+c(\sigma)\mu))+x\cdot \mu^2 \\
      = & E_{\tilde{\pi}}([\tilde{\pi}^\flat])+E^\prime_{\tilde{\pi}}([\tilde{\pi}^\flat])[\tilde{\pi}^\flat]c(\sigma)\mu+y\cdot \mu^2
    \end{array}
  \]
  for some $x,y\in A^\prime_\inf$. 
As $\mu$ is divisible by $E_{\tilde{\pi}}([\tilde{\pi}^\flat])$, we can divide this equation by $E_{\tilde{\pi}}([\tilde{\pi}^\flat])$ and then apply $\theta^\prime$. 
We conclude that
  \[
    \theta^\prime(\frac{E_{\tilde{\pi}}(\sigma([\pi^\flat]))}{E_{\tilde{\pi}}([\tilde{\pi}^\flat])})=1+E^\prime_{\tilde{\pi}}(\tilde{\pi})\theta^\prime(\frac{\mu}{E_{\tilde{\pi}}([\tilde{\pi}^\flat])})\tilde{\pi}c(\sigma)
  \]
  because $\theta^\prime(\mu)=0$. 
By definition of $z$ we can conclude the lemma.
\end{proof}

We now specialize to the case that $K=K^\prime, C=C^\prime=\widehat{\overline{K}}$ and $\tau=\mathrm{Id}_{\mathrm{K}}$. 
In this case, let
\[
  G_K:=\Gal(\overline{K}/K)
\]
be the Galois group of $K$, which acts on $C$ by continuous automorphisms. 
As the morphism
\[
  \tilde{f}\colon \Spf(\mathcal{O}_C)\to \mathrm{Spf}(\mathcal{O}_K)^{\rm HT}
\]
is $G_K$-equivariant (by naturality of the Hodge--Tate stack), the $G_\pi$-torsor 
\[
  Z_\pi\to \Spf(\mathcal{O}_C)
\]
aquires an action of $G_K$. 
From \Cref{sec:galo-acti-cart-1-description-of-sigma-z} we can deduce an explicit description of this action. 
To formulate it, let us first make a definition.

\begin{definition}
  \label{sec:expl-funct-chi-pi-flat}
  We set
  \[
    \chi_{\pi^\flat}\colon G_K\to \mathbb G_m^\sharp(\mathcal{O}_C)=1+p^\alpha\cdot \O_C , \ \sigma\mapsto 1+e \cdot z\cdot c(\sigma)=\theta\left(\frac{\sigma(E_\pi([\pi^\flat]))}{E_\pi([\pi^\flat])}\right),
  \]
  (here $\alpha=\frac{1}{p-1}$ if $p\neq 2$ and $\alpha=1/2$ if $p=2$ defines the convergence radius for the exponential) where the last equality follows from \Cref{sec:galo-acti-cart-2-description-of-quotient-of-eisenstein-polynomials-of-uniformizers}.
\end{definition}

We recall that the choice $\pi^\flat$ of $p$-power compatible roots of $\pi$ yields an isomorphism
\[
  \gamma_{\pi^\flat}\colon G_{\pi,\mathcal{O}_C}\cong Z_\pi.
\]

As a corollary of the previous proposition, we obtain:

\begin{corollary}
  \label{sec:galo-acti-cart-1-galois-action-on-a-en-0}
  The isomorphism $\gamma_{\pi^\flat}\colon G_{\pi,\mathcal{O}_C}\to Z_\pi$ is equivariant for the $G_K$-action if $\sigma\in G_K$ acts on $G_{\pi,\mathcal{O}_C}$ by
  \[
    G_{\pi,\mathcal{O}_C}\xrightarrow{\cdot d_\sigma} G_{\pi,\mathcal{O}_C}\xrightarrow{\mathrm{Id}_{G_{\pi}}\times \sigma} G_{\pi,\mathcal{O}_C},
  \]
  where $d_\sigma:=(\chi_{\pi^\flat}(\sigma),z\cdot c(\sigma))^{-1}$.
\end{corollary}
The map $d\colon G_K\to G_{\pi}(\mathcal{O}_C),\ \sigma\mapsto d_\sigma$ is a $1$-cocycle, and we can formulate Corollary~3.4 by saying that the $G_K$-action on $G_{\pi,\mathcal{O}_C}$ induced by the natural action on $Z_\pi$ via transport of structure through $\gamma_{\pi^\flat}\colon G_{\pi,\mathcal{O}_C}\to Z_\pi$ is given by twisting the usual action on coefficients by $d$.
\begin{proof}
  This follows by setting $\tilde{\pi}^\flat=\pi^\flat$ in \Cref{sec:galo-acti-cart-1-description-of-sigma-z}, as in this case $\tilde{e}=e$.
\end{proof}

\begin{remark}
	The object $Z_\pi$ is also implicit in the work of Min-Wang when they consider the coproduct of the prisms $(A_\pi,I_\pi)$ and $A_\inf(\O_C)$ in the prismatic site of $\O_K$, see \cite[ Lemma 2.11 and \S 5]{MinWang22}. But it plays a different role in their work, as they use it to study $\tau$-connections. For us, $\O(Z_\pi)$ will play the role of a period ring in the style of Fontaine, as we explain in the following.
\end{remark}

\subsection{Comparison to Colmez' ring $B_{\Sen}$}
\label{sec:comp-colm-ring-comparison-to-b-sen}
We will continue with the notation from \Cref{sec:expl-funct-explicit-functoriality} in the situation that $K=K^\prime, C=C^\prime=\widehat{\overline{K}}, \tau=\mathrm{Id}_K, \tilde{\pi}^\flat=\pi^\flat$.
Let $A_{\en}:=\O(G_{\pi,\O_C})$ and set $B_{\en}:=A_{\en}[\tfrac{1}{p}]$.

  The map
  \[
    G_{\pi,\O_C}\to \Ga^\sharp,\ (t,a)\to a
  \]
  is an isomorphism of formal schemes. 
In particular, we have an identification
  \[
    A_\en\cong \widehat{\bigoplus\limits_{n\geq 0}} \O_C\cdot \frac{a^n}{n!}.
  \]
  As shown in \Cref{sec:galo-acti-cart-1-galois-action-on-a-en-0}, on the element $t:=1+e\cdot a\in A_\en$ the $G_K$-action is given by
  \[
    \sigma(t)=\chi_{\pi^\flat}(\sigma)^{-1}\cdot t,
  \]
  with $\chi_{\pi^\flat}(\sigma) \in 1+p^\alpha\cdot \O_C$ as defined in \Cref{sec:expl-funct-chi-pi-flat}. 
Hence the $1$-cocycle $\chi_{\pi^\flat}$ plays a similar role in our setting as the cyclotomic character $\chi\colon G_K\to \Z_p^\times$ plays in $p$-adic Hodge theory via the cyclotomic tower (in the sense that on the usual element $t=[\varepsilon]-1$, the action of $G_K$ is via $\chi$).

To make this analogy more precise, let us first introduce a ``cyclotomic''
version of $B_{\en}$ by setting
\[ B_{\en}^{\cycl}=B_{\en,K}^{\cycl}:=\widehat{\bigoplus\limits_{n\geq 0}}\O_C\cdot  \frac{1}{n!}\left(\frac{t'-1}{e}\right)^n[\tfrac{1}{p}].\]
We note that this Banach algebra over $C$ depends on $K$, but not on the choice of $\pi$. 
Indeed, the implicit orthonormal basis $(\frac{1}{n!}\left(\frac{t'-1}{e}\right)^n)_{n\in \N}$ depends on $e=E^\prime_\pi(\pi)$, but the ideal $e\cdot \O_K$ is the different of $K$ over $W(k)[1/p]$ and thus independent of $\pi$.

The $C$-Banach algebra $B_{\en}^{\cycl}$ admits a (continuous) Galois action by $G_K$ via \[\sigma(t'):=\chi(\sigma)t',\]
where $\chi$ denotes the cyclotomic character.

We now wish to compare $B_{\en}$ and $B_{\en}^{\cycl}$, so we need to find an appropriate coordinate transformation from $t$ to $t'$ that transforms the $1$-cocycle $\chi_{\pi^\flat}$ to the cyclotomic character. 
For this we use:

\begin{lemma}\label{l:galois-action-on-z}
	For any $\sigma\in G_K$,
	\[\frac{\sigma(z)}{z}=\chi(\sigma)\chi_{\pi^\flat}(\sigma)^{-1}.\]
      \end{lemma}
      Here, the element $z=\pi\theta(\frac{\mu}{E_\pi([\pi^\flat])})$ was introduced before~\Cref{sec:galo-acti-cart-1-description-of-sigma-z}.
\begin{proof}
	Since $\pi$ is fixed by the action of $\sigma$, we have $\sigma(z)/z=\sigma(z_0)/z_0$ for $z_0:=\theta(\frac{\mu}{E_{\pi}([\pi^\flat])})$. 
Observe now that
	\[ \frac{\sigma(z)}{z}\cdot \chi_{\pi^\flat}(\sigma)=\theta\left(\frac{\sigma(\frac{\mu}{E_{\pi}([\pi^\flat])})}{\frac{\mu}{E_{\pi}([\pi^\flat])}}\right)\cdot \theta\left(\frac{\sigma({E_{\pi}([\pi^\flat])})}{E_{\pi}([\pi^\flat])}\right)=\theta\left(\frac{\sigma(\mu)}{\mu}\right)=\chi(\sigma)\]
        as desired.
\end{proof}
Focusing just on the Galois action, it therefore looks promising to send
\[ t=1+e\cdot a\mapsto \frac{z}{t'}\]
since \[\sigma(\frac{z}{t'})=\frac{z\chi(\sigma)\chi_{\pi^\flat}(\sigma)^{-1}}{\chi(\sigma)t'}=\chi_{\pi^\flat}(\sigma)^{-1}\frac{z}{t'}\]
transforms as desired to get Galois equivariance of the transformation. 
However, this is problematic due to convergence issues: in order for this rule to extend to a morphism
\[  B_{\en}\rightarrow B^{\cycl}_{\en},\]
the sum 
\[\sum \frac{1}{n!}\left(\frac{\frac{z}{t'}-1}{e}\right)^n\]
would have to converge in  $ B^{\cycl}_{\en}$, or in other words we would have to have \[\frac{z}{t'}\in 1+e\cdot p^{\alpha}\O_C\subseteq 1+p^\alpha \mathcal{O}_C\]
for all $t'\in \mathbb G_m^\sharp(\O_C)=1+p^{\alpha}\O_C$ where as before, $\alpha$ is the convergence radius of the exponential ($\alpha=\frac{1}{p-1}$ or $\alpha=\frac{1}{2}$ if $p=2$). 
This is equivalent to requiring
\[ z\in 1+p^\alpha\O_C,\]
which is not necessarily true: indeed, specializing $z$ along the map $\O_C\to \overline{k}$, we can see that $z$ need not lie in $1+\mathfrak{m}_C$ if the constant term of $E_\pi$ is not $p$.

This raises the question whether one can always modify $z$ so that it is of the desired form: Suppose that we have a different $z'\in C^\times$ for which $\sigma(z')/z'=\chi(\sigma)\chi_{\pi^\flat}(\sigma)^{-1}$ as in \Cref{l:galois-action-on-z}, but also $z'\in 1+ep^\alpha\O_C$. 
Then we have $\sigma(z/z')=z/z'$, so $z'=b z$ for some $b\in K^\times$. 
This shows that $z'$ is of the desired form if and only if 
\[ z'\in (1+ep^{\alpha}\O_C)\cap (z\cdot K^\times).\]
Such an element might not exist, but  due to the fact that $\overline{K}^\times\subseteq C^\times$ is dense, we can always find such a $z'$ after replacing $K$ by a finite extension $K'$. 
Crucially, however, there is in this case not a unique element $z'$, but rather there is an ambiguity by a factor $1+ep^\alpha\O_C\cap K'$. 
In summary:

\begin{proposition}\label{p:comparison-B_en-B_Sen}
	There exists  $z'\in C^\times$ such that there is an $\O_C$-linear continuous isomorphism
	\[ B_{\en}\xrightarrow{\sim}B_{\en}^{\cycl},\quad t=1+e\cdot a\mapsto \frac{z'}{t'}\]
	which is $G_{K'}$-equivariant for some finite extension $K'|K$.
	The element $z'$ is in general non-canonical and only well-defined up to a factor in $1+e\cdot p^{\alpha}\O_{K'}$.
      \end{proposition}
      \begin{proof}
        This follows from our above discussion.
      \end{proof}
\begin{remark}
There is a slightly different perspective on this which also allows us to make $K'$ slightly more precise: Consider $\chi\cdot \chi_{\pi^\flat}^{-1}:G_K\to C^\times$ as a continuous $1$-cocycle which defines a trivial class in $H^1_{}(G_K,C^\times)$ by \Cref{l:galois-action-on-z}. 
For any $k\in \N$, there is $m\in \N$ (by continuity) such that for $K':=K(\pi^{1/p^m},\zeta_{p^m})$, this cocycle restricts on $G_{K'}$ to a cocycle with image in $1+p^{k}\O_C$. 
Consider the commutative diagram
\[\begin{tikzcd}
	{H^1_{}(G_{K'},1+p^{k+1}\O_C)} \arrow[d, "\frac{1}{p^{k+1}}\log","\sim"'] \arrow[r] & {H^1_{}(G_{K'},1+p^{k}\O_C)} \arrow[d, "\frac{1}{p^k}\log","\sim"'] \arrow[r] & {H^1_{}(G_{K'},C^\times)} \arrow[d, "\log",hook'] \\
	{H^1_{}(G_{K'},\O_C)} \arrow[r, "\cdot p"] & {H^1_{}(G_{K'},\O_C)} \arrow[r, "\cdot p^k"] & {H^1_{}(G_{K'},C)}
      \end{tikzcd}\]
as we will recall in \Cref{p:Tate-thm-integral-version} below that the torsion in $H^1_{}(G_{K'},\O_C)$ is uniformly bounded among all finite extensions $K'$ of $K$. 
Since $\chi\cdot \chi_{\pi^\flat}^{-1}$ defines an element in the top left that vanishes in the top right,  it follows that for $k\gg 0$, the class of $\chi\cdot \chi_{\pi^\flat}^{-1}$ in $H^1_{}(G_K,1+p^{k}\O_C)$ vanishes. 
Thus there exists $z'\in K'$ whose coboundary equals $\chi\cdot \chi_{\pi^\flat}^{-1}$.

This shows that we can in fact choose $K'$ to be $K(\pi^{1/p^m},\zeta_{p^m})$ for some $m\gg 0$. 
From this perspective, the reason that $z'$ is non-canonical is that we can only prove the vanishing of a class in group cohomology, but we did not produce a canonical coboundary representing it.
\end{remark}

To reformulate Sen theory in the style of Fontaine, Colmez, \cite{colmez1994resultat}, has introduced the period ring $B_{\Sen}$. 
As a ring, $B_\Sen$ identifies with the ring $C\{\!\{u\}\!\}$  of power series in $u$ with coefficients in $C$ having a positive radius of convergence\footnote{Note that $u\in B_\Sen$ is not related to the element $u\in A_\pi$ from \Cref{sec:expl-funct-explicit-functoriality}.}. 
It is filtered by the subrings $B_\Sen^n$, $n\geq 0$, defined as the power series with radius of convergence $\geq p^{-\max(1,n)}$ for $p$ odd and $p^{-\max(2,n)}$ for $p=2$. 
If for $\sigma \in G_K$ we set
\[ \sigma(u) = u+ \log(\chi(\sigma)),
\]
one checks that this defines an action of $G_{K_n}$ on $B_\Sen^n$, where $K_n=K(\mu_{p^n})$. 
Note that there is an embedding 
\begin{equation}\label{eq:emb-B_en->B_Sen}
B_{\en}^{\cycl}\hookrightarrow B_{\Sen},\quad t'\mapsto \exp(u).
\end{equation}
which factors through $B^1_{\Sen}$ and is Galois equivariant with respect to the action on $B^1_{\Sen}$.

Therefore, as a consequence of \Cref{p:comparison-B_en-B_Sen}, we deduce:

\begin{corollary}\label{p:comparison-B_en-B_Sen-bis}
	There exists  $z'\in C^\times$ such that there is an $\O_C$-linear continuous embedding
	\[ B_{\en} \hookrightarrow B_{\Sen},\quad t=1+e\cdot a\mapsto \frac{z'}{\exp(u)}\]
	that factors through $B^1_{\Sen}$ and is $G_{K'}$-equivariant for some finite extension $K'|K$.
	The element $z'$ is in general non-canonical and only well-defined up to a factor in $1+e\cdot p^{\alpha}\O_{K'}$.
      \end{corollary}

\begin{remark}
	The subscript ``$\en$'' in $B_{\en}$ stands for ``enhanced'', the motivation of which is two-fold: First, the notation is to suggest that $B_{\en}$ is closely related to $B_{\Sen}$, but has a finer convergence condition. 
As we will explain in \Cref{eq:Galois-cohom-Ben-vs-BSen} below, this makes $B_{\en}$ better behaved with respect to Galois cohomology.
	Second, the ring $B_{\en}$ has a natural generalisation for $X^{\rm HT}$ when $X$ is a smooth formal scheme over $\O_K$, and in this case it encodes a close relation with the functor from v-vector bundles to the ``enhanced Higgs modules'' of  Min--Wang \cite{MinWang22}, which are roughly a hybrid of Sen modules and Higgs modules.
\end{remark}

We end this subsection with a slightly different viewpoint on the ring $B_\en$.
As we saw above, the $G_K$-action on
\[
  t^\prime:=\frac{z}{t}
\]
is via the cyclotomic character. 
In particular, Fontaine's period ring
\[
  B_{\mathrm{HT}}:=C[(t^\prime)^{\pm 1}]\cong \bigoplus\limits_{n\in \Z} C(n)
\]
embeds as a subring into $B_\en$. 
In fact, $B_\en$ is the completion of $B_{\mathrm{HT}}$ for the Banach norm induced by the $\mathcal{O}_C$-lattice $A_\en\cap B_{\mathrm{HT}}$, and this sublattice can be made rather explicit.

Let us give an explicit consequence of this completion.

\begin{example}
  \label{sec:c_k-semil-galo-2-example-for-non-split-extension} The class of continuous semilinear $C$-representations $V$ of $G_K$, which satisfy
  \[
  \dim_K (B_\en \otimes_C V)^{G_K} = \dim_C V
  \]
  is different from the class of Hodge--Tate representations. 
As a concrete example, let $M$ be a free rank $2$ vector space over $K$ and equip $M$ with the Sen operator \[
   \Theta_{\pi,M}:=
  \begin{pmatrix}
    0 & -1 \\
    0 & 0
  \end{pmatrix}
\]
for some choice of basis. As we explain in detail in \S4, the $G_K$-representation associated to $(M,\Theta_\pi)$ via $\alst_K \circ\beta_\pi^{-1}$ is explicitly given by $V:=(M\otimes_{\mathcal{O}_K}B_{\en})^{\Theta_\pi+\Theta_{\pi,M}=0}$.
  Firstly, the constants $C$ lie in $V$. 
To get the whole of $V$ we have to solve the differential equation
  \[
    (1+ea)\frac{\partial}{\partial a}(f(a))=1,
  \]
  which yields that $f(a)=\frac{1}{e}\log(1+ea)+c$ for some constant $c\in C$, and thus $V=\langle 1, \log(1+ea)\rangle_C$.
  In particular, $V$ is $2$-dimensional and verifies $\dim_K (B_\en \otimes_C V)^{G_K}=2$. But $V$ is not Hodge--Tate as the Sen operator is not semi-simple.
\end{example}

\subsection{Calculation of Galois cohomology of $B_\en$}
\label{sec:calc-galo-cohom}

We continue to use the notation of \Cref{sec:comp-colm-ring-comparison-to-b-sen}. 
The following theorem is the crucial input in our proof of the full faithfulness part of \Cref{main-theorem-introduction}.2 in \Cref{sec:c_k-semil-galo-1-pullback-fully-faithful-up-to-quasi-isogeny} below.

\begin{theorem}\label{t:cohom-B_en}
  \label{sec:c_k-semil-galo-3-continuous-galois-cohomology-of-b-en-0}
  The natural map
  \[
    K\to R\Gamma(G_K, B_\en)
  \]
  is an isomorphism.
\end{theorem}
Here, the right hand side is continuous group cohomology for the Banach space topology on $B_\en$. 
Before starting the proof, let us recall Tate's celebrated theorem (\cite[Proposition 8(a)]{tate1967p}) which will be used in the proof:
\[
R\Gamma(G_K,C)=K \oplus K[-1].
\]
We will also need the following result, which yields  a more precise version of the result that $H^1(G_K,C(n))=0$ for $n\neq 0$: For the formulation,	let $K_\infty\subseteq C$ be the completed $p$-power cyclotomic extension of $K$ and let $\Gamma:=\Gal(K_\infty|K)$. We fix a topological generator $g\in \Gamma$.
  \begin{lemma}[{\cite[Proposition 7(c) and its proof]{tate1967p}}]\label{p:Tate-thm-integral-version} 
  	There is a functorial $\Gamma$-equivariant $K$-linear decomposition $K_\infty=K\oplus X$ in $K$-Banach spaces such that $X$ has the following property:
  	For any $0\neq n\in \Z$, the $K$-linear operator $g \chi(g)^n-1:X\to X$ has a bounded inverse $\rho_n$. 
More precisely, there is $\delta>0$ such that $|\rho_n(x)|\leq p^{\delta}|x|$ for each $x\in X$. 
This $\delta$ only depends on $K$, and in fact we can choose $\delta$ such that it works uniformly for all finite extensions $K(\mu_{p^m})|K$ for $m\in \N$.
  \end{lemma}
\begin{proof}
	Tate proves this for $g-\lambda$ where $\lambda\in 1+p\Z_p$. 
We can take $\lambda = \chi(g)^{-n}$ and then multiply $g-\lambda$ by $\chi(g)^n$ which does not change the estimates. 
The last statement on independence of $K$ is \cite[Remark on page 172]{tate1967p}.
\end{proof}

\begin{proof}[Proof of \Cref{sec:c_k-semil-galo-3-continuous-galois-cohomology-of-b-en-0}]
  First note that it suffices to prove that there exists a finite extension $L/K$ such that
  \[
   K^\prime \cong R\Gamma(G_{K^\prime},B_{\en,K}^{\cycl}) 
 \]
 for all finite extensions $K^\prime$ of $L$ where $B_{\en,K}^\cycl$ is the Banach $C$-algebra from \Cref{sec:comp-colm-ring-comparison-to-b-sen} and with respect to the field $K$.

 Indeed, given this claim we can use \Cref{p:comparison-B_en-B_Sen} to deduce that
 \[
   K^\prime\cong R\Gamma(G_{K^\prime}, B_{\en})
 \]
 for a finite Galois extension $K'$ over $K$. Let $H=\Gal(K'|K)$, then 
applying $R\Gamma(H,-)$ yields the theorem by Shapiro's lemma as $K^\prime$ is an induced $H$-representation (by the normal basis theorem) with invariants $K$.

 To prove the claim we first set up some notation.
 We set $B:=\widehat{\bigoplus\limits_{n\geq 0}} C \frac{a^n}{n!}$ and $A:=\widehat{\bigoplus\limits_{n\geq 0}} \mathcal{O}_C\frac{a^n}{n!}$. 
We fix some $e\in \mathcal{O}_K \backslash \{0\}$ and set $t:=1+e\cdot a$. 
We equip $B$ with the unique continuous Galois action such that $\sigma(t)=\chi(\sigma)t$ for $\sigma\in G_K$ and $\chi\colon G_K\to \Z_p^\times$ the cyclotomic character. 
(Thus, in particular, if $e=E^\prime(\pi)$ up to $\mathcal{O}_K^\times$, then $B\cong B^\cycl_\en$ from \Cref{sec:comp-colm-ring-comparison-to-b-sen}.)
 We can conclude
 \begin{equation}\label{eq:action-of-Gamma-on-a}
   \sigma(a)=\chi(\sigma)\cdot a+\frac{\chi(\sigma)-1}{e}
 \end{equation}
 for $\sigma\in G_K$ as $t=1+e\cdot a$.

	\subsubsection*{Step 1:  Calculation of $H^0(G_K, B)$}
	
	We start by proving that 
	\[B^{G_K} =K.\]
	For this we can use an analog of the embedding \Cref{eq:emb-B_en->B_Sen} and use \cite[Th\'eor\`eme~2.(i)]{colmez1994resultat}.	
	Alternatively, we can argue directly and follow Colmez' argument for $B_{\Sen}$, as we shall now demonstrate. 
In the following we get slightly different series expansions due to the difference in Galois actions on $u$ and $a$.

Since the action of $G_{K_\infty}:=\Gal_{\mathrm{cts}}(C/K_\infty)$ on $a$ is trivial, we see from the definition of $B$ that 
	\[ \left(\widehat \bigoplus_{n\geq 0} C \frac{a^n}{n!}\right)^{G_K}=\left(\widehat \bigoplus_{n\geq 0} K_\infty \frac{a^n}{n!}\right)^{\Gamma}.\] 
	Via the cyclotomic character $\chi:\Gamma\to \Z_p^\times \to \mathcal{O}_K^\times$, we can identify $\Gamma$ with a closed subgroup of $\mathcal{O}_K^\times$. 
Since it suffices to prove that there is a finite extension $K'|K$ for which $B^{G_{K'}}=K'$, we may for any $k\in \N$ assume after passing to $K'$ large enough that $\chi$ has image in $1+p^k\Z_p$. Hence we may assume that 
\[y:=\frac{1}{e}(\chi(g)-1)\in p\O_K,\]
or $y\in 4\O_K$ if $p=2$.
Let now $g\in \Gamma$ and $b:=\sum b_n\frac{a^n}{n!}\in B$.
Then since $G_K$ acts trivially on $y$, we have
	\begin{eqnarray}\label{eq:explicit-action-of-GK-on-Ben}
		g(\sum_{n\geq 0} b_n\frac{a^n}{n!})&=&\sum_{n\geq 0} g(b_n)\frac{(\chi(g)a+y)^n}{n!}=\sum_{n\geq 0}g(b_n)\sum_{k+m=n}\frac{\chi(g)^ma^m}{m!} \frac{y^k}{k!}\nonumber\\
		&=&\sum_{m=0}^\infty \chi(g)^m\left(\sum_{k=0}^\infty g(b_{m+k})\frac{y^k}{k!}\right)\frac{a^m}{m!}.
	\end{eqnarray}
	Here we note that the sum in brackets converges in $C$ because $y\in p\O_K$ (or $y\in 4\O_K$ for $p=2$).
	 
Assume now that $g(b)=b$ for all $g$, then by comparing coefficients of $\frac{a^n}{n!}$, we see that this is equivalent to
	\[ b_n=\chi(g)^n\left(\sum_{k=0}^\infty g(b_{n+k})\frac{y^k}{k!}\right)\]
        for all $n\geq 0$.
	Following Colmez' argument, we now multiply both sides by $g^{-1}$ and apply for any $j>0$ the normalised trace $\mathrm{tr}_j\colon K_\infty\to K_j$ to the $j$-th cyclotomic subextension $K_j|K$. 
Then since $\mathrm{tr}_j$ is $\Gamma$-equivariant, we have
	\[g^{-1}\mathrm{tr}_j(b_n)=(e\cdot y+1)^n\left(\sum_{k=0}^\infty \mathrm{tr}_j(b_{n+k})\frac{y^k}{k!}\right).\]
	We can consider this expression as a function $\Gamma\to K_j$ of the variable $g\in \Gamma\subseteq \Z_p^\times$ (recall that $y=\frac{1}{e}(\chi(g)-1)$). 
Then the left hand side shows that this function is locally constant, whereas the right hand side shows that this function is analytic. 
It follows that it is constant, hence $\mathrm{tr}_j(b_n)\in K$ for all $j$, hence $b_n\in K$. 
Consequently, for all $n\geq 0$ we have
	\begin{equation}\label{eq:cond-for-b-in-B_en^0-to-be-Galois-inv}
		b_n=(e\cdot y+1)^n\left(\sum_{k=0}^\infty b_{n+k}\frac{y^k}{k!}\right).
	\end{equation}
	For $n=0$, this means that $b_0=\sum_{k=0}^\infty b_{k}\frac{y^k}{k!}$,
	and comparing these inside $C[[y]]$ we deduce that $b_k=0$ for $k\geq 1$. 
Hence $b=b_0\in K$ as desired.

	\subsubsection*{Step 2: Calculation of $H^1(G_K, B)$}
	In order to compute $H^1$, we first wish to see that the natural map
	\begin{equation}\label{eq:H^1-of-C->B_en0}
		H^1_{}(G_K,C)\to H^1_{}(G_K,B)
	\end{equation}
	vanishes. 
To see this, we consider the short exact sequence of $C$-linear $G_K$-modules
\[ 0\to C\to B\to B/C\to 0.\]
  From Tate's theorem we know that
  \[
    R\Gamma(G_K,C)\cong K\oplus K[-1].
  \]
  We claim that also 
  \[ (B/C)^{G_K}\cong K\]
  as a $K$-vector space. 
Once this is known, we see by comparing $K$-vector space dimensions in the long exact sequence
  \[ 0\to \underset{=K}{H^0(G_K,C)}\to \underset{=K,\ \textrm{ by Step 1}}{H^0(G_K, B)}\to H^0(G_K,B/C)\to H^1(G_K,C)\xrightarrow{\eqref{eq:H^1-of-C->B_en0}} H^1(G_K, B)\to H^1(G_K,B/C)\]
  that \eqref{eq:H^1-of-C->B_en0} is zero.
  
  In fact, it suffices to see that $(B/C)^{G_K}\neq 0$ as it must then be of $K$-dimension $1$. 
The non-vanishing is witnessed by the element
  \[
    b:=\log(t)=\sum\limits_{n=1}^\infty(-1)^{n-1}\frac{(e\cdot a)^n}{n}\in B
  \]
  for which $g(\log(t))=\log(t)+\log(\chi(g))\equiv \log(t) \in B/C$. 
Before going on, we note that one can also prove the claim $(B/C)^{G_K}\cong K$ by arguing exactly as in the first part: Observe that as a $C$-module there is a description
  \[\textstyle B/C= \widehat{\bigoplus\limits_{n\geq 1}} C\cdot \tfrac{a^n}{n!}.\]
 	Like for $B$, we deduce that an element $\sum b_n\frac{a^n}{n!}$ in the right hand side satisfies $g(b_n)=b_n$ if and only if \eqref{eq:cond-for-b-in-B_en^0-to-be-Galois-inv} holds, but this time only for $n\geq 1$ instead of $n=0$. 
From this point on, the arguments diverge: We cannot conclude anymore by considering $n=0$, instead we need to expand the right hand side of \eqref{eq:cond-for-b-in-B_en^0-to-be-Galois-inv} for $n=1$, where we find
 	\[ b_1=b_1+\sum_{k=1}^\infty (b_{k+1}+ekb_{k})\frac{y^k}{k!},\]
  and similarly for $n\geq 2$, where the expression is independent of $b_1$. 
This means that we can choose $b_1\in K$ freely, and then the vanishing of the coefficients in front of $y^k$ for $k\geq 1$ impose linear relations on the $b_k$ for $k\geq 2$ which determine $b$ uniquely. 
Hence $\dim_K (B/C)^{G_K}= 1$. 
  \medskip
  
  To finish the proof of the Theorem, it remains by the above long exact sequence to show that
  \[ H^i(G_K,B/C)=0\]
  for $i\geq 1$.
  For this we again first reduce to the extension $K_\infty|K$:
  First note that $H^i(G_{K_\infty},B/C)=\widehat{\bigoplus\limits_{n\geq 1}} H^i(G_{K_{\infty}},C)=0$ for $i\geq 1$ because $G_{K_\infty}$ acts trivially on $a$ and $K_\infty$ is perfectoid. 
Thus it suffices to compute
  \[ H^1(\Gamma,(B/C)^{G_{K_\infty}})\]
  as this is the only group having a potential non-zero contribution to $R\Gamma(G_K,B)$.
  We first note that $(B/C)^{G_{K_\infty}} = \widehat\oplus_{n\geq 1} K_\infty \frac{a^n}{n!}=:D\subseteq B/C$.
  We may by shrinking $\Gamma$ assume that $\Gamma$ is pro-cyclic. 
Then we can choose a topological generator $g\in \Gamma$ and compute $H^1(\Gamma,D)$  as the cokernel of the map
  \[ g-1:D\to D.\]
  It remains to prove that this morphism is surjective. 
To see this, it suffices to prove the statement for $g^m-1$ (for example because $g-1$ divides $g^m-1$, or by inflation-restriction for the closed subgroup generated by $g^m$). 
So we are later free to replace $g$ by $g^m$.
  
  Let now $b=\sum_{n=1}^\infty b_n\frac{a^n}{n!}\in D$, then using \eqref{eq:explicit-action-of-GK-on-Ben}, we have
  \[ (g-1)b=\sum_{n=0}^\infty \left(\chi(g)^n\left(\sum_{k=0}^\infty g(b_{n+k})\frac{y^k}{k!}\right)-b_n\right)\frac{a^n}{n!}.\]
  Let furthermore $c=\sum_{n=1}^\infty c_n\frac{a^n}{n!}\in D$, then to show that $c$ lies in the image we need to solve the equation $(g-1)b=c$, which is equivalent to asking that for all $n\geq 1$, we have
  \[c_n=(g(b_n)\chi(g)^{n}-b_n)+\sum_{k=1}^\infty \chi(g)^ng(b_{n+k})\frac{y^k}{k!}.\]
  In other words, if we consider $(g-1)$ as a continuous endomorphism of the Banach $K_\infty$-module $D$, then in terms of the orthonormal basis $(\frac{a^n}{n!})_{n\geq 1}$, this is represented by an infinite upper triangular block matrix
  
  \[g-1=\left(\begin{matrix}
  		g\chi(g)-1 & \ast &\ast &\dots\\
  		 & g\chi^2(g)-1 &\ast &\dots\\
  		 &  &g\chi^3(g)-1 &\dots\\
  		&&& \ddots
              \end{matrix}\right)\]
          (each block represents a semilinear endomorphism of a $1$-dimensional $K_\infty$-vector space, and the whole matrix a $K$-linear endomorphism). 
          
    To analyse this matrix, we now use the  decomposition $D=D_K\oplus D_X$ according to the decomposition of $K_\infty$ in \Cref{p:Tate-thm-integral-version}, namely we set \[\textstyle D_K:=\widehat\bigoplus_{n\geq 1} K \tfrac{a^n}{n!}\quad \text{and} \quad D_X:=\widehat\bigoplus_{n\geq 1} X \tfrac{a^n}{n!}.\]
    Since the decomposition $K_\infty=K\oplus X$ is $K$-linear and Galois-equivariant, it follows from \eqref{eq:explicit-action-of-GK-on-Ben} that the same is true for $D=D_K\oplus D_X$.
    
    We claim that $g-1$ is even invertible on $D_X$: To this end, we first observe that the entries above the diagonal are all $p$-adically small: Namely, let $M$ be the matrix given by the upper triangular entries, then we can assume that 
	\[|M|\leq |y|.\]
        Indeed, after enlarging $K$ we may assume that $y\in 2p\O_K$, and then $|\frac{y^k}{k!}|\leq |y|$ for all $k\geq 1$.

   	\Cref{p:Tate-thm-integral-version} shows that the $K$-linear operator defined by the block matrix $\rho:=\mathrm{diag}(\rho_1,\rho_2,\dots)$ is bounded and continuous. 
We now consider
  	\[ \rho\circ (g-1)=1+\rho\circ M,\]
  	then $|\rho M|\leq p^\delta |y|$. 
Passing to an extension of $K$ and replacing $g$ by $g^m$, we can make $y=\frac{1}{e}(\chi(g)-1)$ arbitrarily small.
  	 As passing to an extension of $K$ does not change $\delta$, we can therefore assume without loss of generality that $|\rho M|<1$. 
 Thus the matrix $M$ is ``topologically nilpotent''. 
In particular, the matrix $1+\rho M$ is indeed invertible with inverse the bounded operator $\sum_{k=0}^\infty (-\rho M)^k$. 
Thus $g-1$ is invertible on $D_X$, showing that $H^1(\Gamma,D_X)=0$ as desired.
  	
  	It remains to prove that $g-1$ is surjective on $D_K$. Recall from the first part that $g-1$ has kernel $\log(t)$ on $D_K$. Hence we aim to see that $g-1$ is invertible after discarding the first basis vector $\tfrac{a^1}{1}$ on the source. This motivates to consider instead the morphism  \[\textstyle\psi:\widehat\bigoplus_{n\geq 2} K \tfrac{a^n}{n!}\subseteq D_X\xrightarrow{g-1} D_X.\]
  	We claim that $\psi$ is invertible.
  	Let us write $u:=\chi(g)$ to simplify notation. Then with respect to the orthogonal basis $(\frac{a^n}{n!})_{n\geq 2}$ on the source, $\psi$ is now represented by the same matrix as above, but with first column deleted:
  	 \[\arraycolsep=7pt\def\arraystretch{1.5}\psi=\left(\begin{matrix}
  u y	& u \frac{y^2}{2!} &u \frac{y^3}{3!}  &\dots\\
  		 u^2-1 &	u^2 y & u^2 \frac{y^2}{2!} &\\
  		0 &u^3-1 &	u^3 y& \\
  		0&0& & \ddots
  	\end{matrix}\right)\]
  	Since $u\in \O_K^\times$ is a unit and $u^n-1\in \O_K$ is divisible by $y=(u-1)/e$ for any $n$, this shows that the matrix for $\psi/y$ has units on the diagonal, entries of norm $\leq 1$ on the lower subdiagonal, and entries of norm $<1$ on the upper diagonal (after shrinking $\Gamma$ and hence $|y|$ if necessary). It follows that the matrix is invertible, hence $g-1$ is surjective on $D_K$.
  	
  	This finishes the proof of \Cref{sec:c_k-semil-galo-3-continuous-galois-cohomology-of-b-en-0}.
\end{proof}
\begin{remark}
	More axiomatically, in the spirit of the Sen axioms of Berger--Colmez \cite{BergerColmez08}, the above proof works for any field $K$ and with the coefficient field $C$ of $B_{\en}$ replaced by the completed cyclotomic extension $K_\infty|K$ with Galois group $\Gamma$ whenever the following conditions are satisfied:
	\begin{enumerate}
		\item We have Tate's Galois equivariant  normalised traces $\mathrm{tr}:K_\infty\to K$, and
		\item The analogue of \Cref{p:Tate-thm-integral-version} holds, i.e.\ for $\Gamma=\Gal(K_\infty|K)$ we have $H^1(\Gamma,K_\infty)=K$, and for any generator $g\in \Gamma$ and $n\geq 1$, the map $g\chi(g)^n-1:X\to X$ on $X:=\mathrm{ker}(\mathrm{tr})$ has a continuous inverse that is bounded, uniformly in $n$ and for all finite subextensions of $K_\infty|K$.
	\end{enumerate}
	Incidentally, we point out that the element $\alpha$ introduced in \cite[\S4.5]{BergerColmez08} to describe the locally analytic vectors for the Galois action in the Kummer tower can be taken to be our element $(ez)^{-1}$.
\end{remark}
\begin{remark}\label{eq:Galois-cohom-Ben-vs-BSen}
It is natural to ask if the analogue of the Theorem also holds for Colmez' ring $B_{\Sen}^m$ with its natural Galois action by $G_{K_\infty}$\footnote{In the terminology of Colmez: i.e., for each $m\geq 0$, $B_\Sen^m$ has an action of $G_{K_m}$.}, but this is not the case. 
Instead of filtering $B_\Sen$ by the subspaces $B_\Sen^m$, let us filter it by the subspaces
\[
\mathrm{Fil}^m B_\Sen = \left\{ \sum_{n\geq 0} b_n \frac{u^n}{n!}, ~~ |b_n| p^{-\sup(1,m)\cdot n} \to 0 \right\}.
\]
 One has $\mathrm{Fil}^m B_\Sen \subset B_\Sen^{m-1}$ for $m\geq 1$, and $G_{K_{m-1}}$ acts on $\mathrm{Fil}^m B_\Sen$. 
Take an element in $\mathrm{Fil}^{m+1} B_\Sen$, $m\geq 0$. 
Following the same line of argument with $a$ replaced by $u$, we see that since the action on $u$ is given by $g(u)=u+\log\chi(g)$ instead of \eqref{eq:action-of-Gamma-on-a}, we need to replace \eqref{eq:explicit-action-of-GK-on-Ben} by \[g\left(\sum_{n\geq 0} b_n\frac{u^n}{n!}\right)=\sum_{l=0}^\infty \left(\sum_{k=0}^\infty g(b_{l+k})\frac{y^k}{k!}\right)\frac{u^l}{l!}\]
where this time $y:=\log(\chi(g))$. 
The essential difference is that there is no longer a factor of $\chi(g)^l$ before the coefficient of $\frac{u^l}{l!}$. 
Consequently, following the same line of argument, the action of $g-1$ with respect to $(\tfrac{u^n}{n!})_{n\geq 1}$ is now given by a block matrix of the form
\[g-1=\left(\begin{matrix}
	g-1 & \ast &\ast\\
	& g-1 &\ast\\
	&& \ddots
\end{matrix}\right).\]
But this has large cokernel as the cokernel of $g-1$ on $C$ is $H^1(G_{K_m},C)=K_m$. 
The conclusion is that $H^1(G_{K_m},\mathrm{Fil}^m B_{\Sen})$ is also large. 

This is one way in which the ring  $B_{\en}$ is  ``enhanced''.
\end{remark}

\section{$C$-semilinear Galois representations via Hodge--Tate stacks}
\label{sec:c_k-semil-galo}
In this section we prove our main result, \Cref{full-description-of-perf-spd-k}, which gives a description of v-perfect complexes on $\mathrm{Spa}(K)$ in terms of Hodge--Tate stacks. 
Let 
\[
  \Spa(K):=\Spa(K,\mathcal{O}_K)
\]
be the diamond associated to $K$, cf.\ \cite[Definition 15.5]{scholze_etale_cohomology_of_diamonds}, and its v-site
\[
  \Spa(K)_v,
\]
cf.\ \cite[Definition 14.1]{scholze_etale_cohomology_of_diamonds}.
If $\Spa(R,R^+)\in \Spa(K)_v$ is an affinoid perfectoid space over $K$\footnote{Here and in the following we identify the v-site of $K$, which consists of perfectoid spaces $S$ in characteristic $p$ and an untilt $S^\sharp$ over $\Spa(K,\mathcal{O}_K)$, with the site of perfectoid spaces over $\Spa(K,\mathcal{O}_K)$.}, then the natural map
\[
  \pi^\HT_{R^+}\colon\mathrm{Spf}(R^+)^{\rm HT}\to \Spf(R^+)
\]
is an isomorphism, cf.\ \cite[Example 3.12]{bhatt2022prismatization}.
This implies that for any $\Spa(R,R^+)\in \Spa(K)_v$ the morphism $f\colon \Spf(R^+)\to \Spf(\mathcal{O}_K)$ lifts naturally to a map
\[
  \widetilde{f}\colon \Spf(R^+)\to \mathrm{Spf}(\mathcal{O}_K)^{\rm HT}
\]
over $\Spf(\mathcal{O}_K)$.
From here we get a symmetric monoidal, exact functor
\[
  \alstp_K\colon \calPerf(\mathrm{Spf}(\mathcal{O}_K)^{\rm HT})\to \calPerf(\Spa(K)_v),
\]
where the latter category means perfect complexes on $\Spa(K)_v$ for the ``completed structure sheaf''
\[
  (\Spa(R,R^+)\to \Spa(K,\mathcal{O}_K))\mapsto R.
\]
Indeed, by definition
\[
  \calPerf(\mathrm{Spf}(\mathcal{O}_K)^{\rm HT}):=\varprojlim\limits_{\Spec(S)\to \mathrm{Spf}(\mathcal{O}_K)^{\rm HT}} \calPerf(S).
\]
where the limit is taken over the category of all (discrete) rings $S$ with a morphism $\Spec(S)\to \mathrm{Spf}(\mathcal{O}_K)^{\rm HT}$. 
Using the maps $\widetilde{f}\colon \Spf(R^+)=\varinjlim\limits_{n} \Spec(R^+/p^n)\to \mathrm{Spf}(\mathcal{O}_K)^{\rm HT}$, the construction of $\alst_K^{+}$ can now be stated as
\[
  (\mathcal{E}_S)_{\Spec(S)\to \mathrm{Spf}(\mathcal{O}_K)^{\rm HT}}\mapsto ((R\varprojlim\limits_{n} \mathcal{E}_{\Spec(R^+/p^n)\to \mathrm{Spf}(\mathcal{O}_K)^{\rm HT}})[1/p])_{\Spa(R,R^+)\to \Spa(K,\mathcal{O}_K)}.
\]
By construction, the functor $\alstp_K$ induces a functor
\[
  \alst_K \colon \calPerf(\mathrm{Spf}(\mathcal{O}_K)^{\rm HT})[1/p]\to \calPerf(\Spa(K)_v)
\]
by passing to the isogeny category on the source.

\begin{remark}
  \label{sec:c_k-semil-galo-2-generic-fiber-of-Hodge--Tate-stack}
  The generic fiber $\mathrm{Spf}(\mathcal{O}_K)_{\eta}^{\rm HT}$ of the Hodge--Tate stack $\mathrm{Spf}(\mathcal{O}_K)^{\rm HT}$ can be defined as the (analytic sheafification of the) functor
  \[
    (B,B^+)\mapsto \varinjlim\limits_{B_0\subseteq B^+ \textrm{ ring of definition}}\mathrm{Spf}(\mathcal{O}_K)^{\rm HT}(B_0) 
  \]
  on complete adic Huber pairs over $(K,\mathcal{O}_K)$, cf.\ \cite[Proposition 2.2.2.]{Scholze2013}. 
Ideally, the category $\calPerf(\mathrm{Spf}(\mathcal{O}_K)^{\rm HT})[1/p]$ should be defined as a category of perfect complexes on $\mathrm{Spf}(\mathcal{O}_K)_{\eta}^{\rm HT}$. 
By \cite[Theorem 7.8]{mathew2022faithfully} one can by descent define such a category of perfect complexes, and using Cartier-duality over $\mathcal{O}_K$ for the Hopf algebra $\mathcal{O}_{G_\pi}$ (and \Cref{sec:complexes-Hodge--Tate-1-complexes-on-ht-locus-for-r}) see that each perfect complex extends to $\mathrm{Spf}(\mathcal{O}_K)^{\rm HT}$ so that the resulting category of perfect complexes should be equivalent to our more concrete version $\calPerf(\mathrm{Spf}(\mathcal{O}_K)^{\rm HT})[1/p]$. 
\end{remark}

By v-descent of perfect complexes on perfectoid spaces, cf.\ \cite[Theorem 2.1]{anschutz2021fourier}, the category $\calPerf(\Spa(K)_v)$ identifies with the category of ``continuous, semilinear representations of $G_K$ on perfect complexes of $C$-vector spaces'' (here $C=\widehat{\overline{K}}$ as before).\footnote{To make the continuity precise, one could use the solid formalism, or just use the $\infty$-limit
  \[
    \varprojlim\limits_{\Delta} \calPerf(\Spa(C,\mathcal{O}_C)^{\bullet+1})\cong \varprojlim\limits_{\Delta} \calPerf(\Spa(C,\mathcal{O}_C)\times \underline{G_K}^\bullet).
  \]}
Let us note that the usual truncations equip $\calPerf(\Spa(K)_v)$ with a $t$-structure whose heart is the usual category of continuous semilinear representations of $G_K$ on finite dimensional $C$-vector spaces.

Our next aim is to prove the following result.

\begin{theorem}
  \label{sec:c_k-semil-galo-1-pullback-fully-faithful-up-to-quasi-isogeny}
The functor $\alst_K \colon \calPerf(\mathrm{Spf}(\mathcal{O}_K)^{\rm HT})[1/p]\to \calPerf(\Spa(K)_v)$ is fully faithful.
\end{theorem}
\begin{remark}
Given \Cref{sec:complexes-Hodge--Tate-complexes-on-ht-divisor-generated-by-bk-twists} and Tate's calculation of the continuous Galois cohomology $H^\ast(G_K,C(n))$, $n\in \Z$, \Cref{sec:c_k-semil-galo-1-pullback-fully-faithful-up-to-quasi-isogeny} might first seem trivial to prove: first, reduce checking fully faithfulness to the case of morphisms between Breuil--Kisin twists $\mathcal{O}_{\mathrm{Spf}(\mathcal{O}_K)^{\rm HT}}\{n\}$, and then apply Tate's theorem. 

However, this strategy does not work: the category $\calPerf(\mathrm{Spf}(\mathcal{O}_K)^{\rm HT})[1/p]$ is not generated under colimits by Breuil--Kisin twists.\footnote{This is not a contradiction as the generation in $\mathcal{D}(\mathrm{Spf}(\mathcal{O}_K)^{\rm HT})$ involves \textit{infinite} colimits.} To remedy this one could hope to prove a more general statement involving a certain rationalized version of $\mathcal{D}(\mathrm{Spf}(\mathcal{O}_K)^{\rm HT})$. 
However, one then also needs to enlarge $\calPerf(\Spa(K)_v)$ accordingly in such a way that $\alst_K$ extends to a functor, which still commutes with arbitrary colimits in $D(\mathrm{Spf}(\mathcal{O}_K)^{\rm HT})$. 
We do not know if this possible. 
This is why we pursue a different strategy and use \Cref{t:cohom-B_en}.
\end{remark}

We will start with some preparations.
Fix a uniformizer $\pi\in \mathcal{O}_K$ and consider the $G_\pi$-torsor
\[
  Z_\pi:=\Spf(A_\en)\to \Spf(\mathcal{O}_C)
\]
as in \Cref{sec:expl-funct-explicit-functoriality}. 
The $\mathcal{O}_C$-algebra $A_\en$ is equipped with a continuous $G_K$-action and a commuting endomorphism $\Theta=\Theta_{A_\en}$. 
Let $(A_\inf,J)$ be the perfect prism associated to $\mathcal{O}_C$. 
After fixing a choice
\[
  \pi^\flat:=(\pi, \pi^{1/p},\ldots)\in \mathcal{O}_C^\flat
\]
of compatible $p$-power roots of $\pi$, which we use to send $u\in A_\pi$ to $[\pi^\flat]$, the operator $\Theta$  and the Galois action have been made explicit in \Cref{sec:complexes-Hodge--Tate-sen-operator-for-regular-representation} and in \Cref{sec:expl-funct-explicit-functoriality}, respectively.
Note that under the equivalence from \Cref{sec:complexes-Hodge--Tate-1-complexes-on-ht-locus-for-r} the pullback along the map
\[
  [\Spf(\mathcal{O}_C)/G_K]\to \mathrm{Spf}(\mathcal{O}_K)^{\rm HT}
\]
identifies with the map
\[
 (M, \Theta_{e,M})\mapsto \mathrm{fib}(M\widehat{\otimes}_{\mathcal{O}_K} A_\en\xrightarrow{1\otimes \Theta_{A_\en}+\Theta_{e,M}\otimes 1} M\widehat{\otimes}_{\mathcal{O}_K} A_\en), 
\]
where $G_K$-acts via $A_\en$ on the fiber. 
As in \Cref{sec:expl-funct-explicit-functoriality} we set
\[
  B_\en:=A_\en[1/p].
\]

Given \Cref{sec:c_k-semil-galo-3-continuous-galois-cohomology-of-b-en-0} we can conclude the proof of \Cref{sec:c_k-semil-galo-1-pullback-fully-faithful-up-to-quasi-isogeny}.

\begin{proof}[Proof of \Cref{sec:c_k-semil-galo-1-pullback-fully-faithful-up-to-quasi-isogeny}]
  Let $\mathcal E\in \calPerf(\mathrm{Spf}(\mathcal{O}_K)^{\rm HT})$ and let $(M,\Theta_{\pi,M})=\beta^+_{\pi}(\mathcal E)$.
Then $\alst_K(\mathcal E)$ is the perfect complex of $C$-vector spaces with $G_K$-action
  \[
    V:=\mathrm{fib}(M\otimes_{\mathcal{O}_K}B_\en\xrightarrow{1\otimes \Theta+\Theta_{\pi,M}\otimes 1}M\otimes_{\mathcal{O}_K}B_\en).
  \]
  It suffices to show that the natural map
  \[
    R\Gamma(\mathrm{Spf}(\mathcal{O}_K)^{\rm HT},\mathcal E)[\tfrac{1}{p}]\to R\Gamma(G_K, V)
  \]
  is an isomorphism. 
Here, the left hand side identifies (by \Cref{sec:complexes-Hodge--Tate-cohomology-on-Cartier--Witt-stack}) with
  \[
    \mathrm{fib}(M\xrightarrow{\Theta_{\pi,M}}M)[\tfrac{1}{p}].
  \]
  For the right hand side we use \Cref{sec:c_k-semil-galo-3-continuous-galois-cohomology-of-b-en-0}. 
Namely,
  \[
    \begin{matrix}
      & R\Gamma(G_K, V)\\
      \cong & R\Gamma(G_K, \mathrm{fib}(M\otimes_{\mathcal{O}_K}B_\en\xrightarrow{1\otimes \Theta+\Theta_{\pi,M}\otimes 1}M\otimes_{\mathcal{O}_K}B_\en))\\
      \cong & \mathrm{fib}(M\otimes_{\mathcal{O}_K} R\Gamma(G_K, B_\en)\xrightarrow{1\otimes \Theta+\Theta_{\pi,M}\otimes 1} M\otimes_{\mathcal{O}_K} R\Gamma(G_K, B_\en))\\
      \cong & \mathrm{fib}(M\xrightarrow{\Theta_{\pi,M}} M)[\tfrac{1}{p}],
    \end{matrix}
  \]
  where we used $R\Gamma(G_K,B_\en)\cong K$ (with $\Theta=0$) in the last step.
\end{proof}

We now turn to the description of the essential image of the functor
\[
  \calPerf(\mathrm{Spf}(\mathcal{O}_K)^{\rm HT})[1/p]\to \calPerf(\Spa K_v,\O).
\]
As the target admits a $t$-structure with heart the abelian category $\Rep_C(G_K)$ the claim reduces to the question when a given $V\in \Rep_C(G_K)$ lies in the image. 
Following \cite{min2021hodge}, \cite{gao2022hodge} the essential image should be given by the nearly Hodge--Tate representations.

\begin{definition}[{\cite[Definition 1.1.3]{gao2022hodge}}]
  \label{sec:c_k-semil-galo-2-nearly-Hodge--Tate}
  A representation $V\in \Rep_C(G_K)$ is nearly Hodge--Tate if its Sen operator $\Theta$ has eigenvalues in $\Z+\delta_{\mathcal{O}_K/\Z_p}^{-1}\mathfrak{m}_{\overline{K}}\subseteq \overline{K}$ where $\delta_{\mathcal{O}_K/\Z_p}=(E^\prime_\pi(\pi))$ is the inverse different ideal and $\Theta\colon V\to V$ denotes the classical Sen operator as constructed in \cite{sen1980continuous}.
\end{definition}

To see this, we start with the following lemma:

\begin{lemma}
	\label{sec:c_k-semil-galo-1-alternative-description-of-galois-action}
	Let $(M,\Theta_{\pi,M})\in \mathcal{D}(BG_\pi)\cong \mathcal{D}(\mathrm{Spf}(\mathcal{O}_K)^{\rm HT})$. 
	Then its pullback to $\Spf(\O_C)$ is isomorphic to the complex $M\widehat{\otimes}_{\O_K}\O_C$ and if $\sigma\in G_K$, then the resulting action of $\sigma$ on $M\widehat{\otimes}_{\O_K}\O_C$ is given by the composition
	\[
	M\widehat{\otimes}_{\O_K}\O_C\xrightarrow{\Id\otimes \sigma} M\widehat{\otimes}_{\O_K}\O_C\xrightarrow{\gamma_{d(\sigma),M}^{-1}} M\widehat{\otimes}_{\O_K}\O_C
	\]
	with $d(\sigma)\in G_\pi(\O_C)$ defined in \Cref{sec:galo-acti-cart-1-galois-action-on-a-en-0} and $\gamma_{d(\sigma),M}$ is as described in \Cref{sec:complexes-Hodge--Tate-explicit-action}. Thus, explicitly
	\[
	\sigma(m\otimes c)=\chi_{\pi^\flat}(\sigma)^{\Theta_{\pi,M}/e}(m\otimes 1)\sigma(c).
	\]
\end{lemma}

\begin{proof}
	First of all note that by \Cref{sec:galo-acti-cart-1-galois-action-on-a-en-0} the Galois action of $\sigma\in G_K$ on $A_\en\cong \mathcal{O}_{G_\pi}\widehat{\otimes}_{\mathcal{O}_K}\mathcal{O}_C$ is given by first applying $\mathrm{Id}_{\mathcal{O}_{G_\pi}} \otimes \sigma$, and then $\gamma_{d(\sigma), \mathcal{O}_{G_\pi}\widehat{\otimes}_{\mathcal{O}_K}\mathcal{O}_C}$. 
	Now, we can apply the following general observation: If $H$ is any group (in any topos) and $V$ any $H$-representation with underlying abelian group $|V|$, then the map
	\[
	H\times V\to H\times |V|,\ (h,v)\to (h,h^{-1}v)
	\]
	is an isomorphism, which converts the diagonal action on the left hand side to the left multiplication on $H$ on the right hand side, and the right multiplication on $H$ in $H\times V$ to the $H$-action on $H\times |V|$ given by $g\cdot (h,v)=(hg,h^{-1}v)$. 
	In particular, the invariants of the diagonal action on $H\times V$ identify with $|V|$ and the remaining right action of $H$ with the inverse of the given action on $V$. 
	Now we can apply this to $H=G_\pi, V=M$ and to the right multiplication on $H$. 
	Formulated at the level of comodules, this implies that
	\[
	M\widehat{\otimes}_{\mathcal{O}_K} \mathcal{O}_C\cong \mathrm{fib}(M\widehat{\otimes}_{\mathcal{O}_K} A_\en\xrightarrow{1\otimes \Theta_{A_\en}+\Theta_{e,M}\otimes 1} M\widehat{\otimes}_{\mathcal{O}_K} A_\en)
	\]
	and that the right $G_\pi$-action on $A_\en\cong \mathcal{O}_{G_\pi}\widehat{\otimes}_{\mathcal{O}_K} \mathcal{O}_C$ transforms to the inverse of the $G_\pi$-action on $M$. 
	This implies the claim.
\end{proof}

For example, if $(M,\Theta_{\pi,M})=(\mathcal{O}_K, e\cdot n)$, then with $z\in C^\times$ as in \Cref{l:galois-action-on-z}
\[
\sigma(1\otimes z^n)=\chi_{\pi^\flat}(\sigma)^{n}\cdot  1\otimes \chi(\sigma)^n\chi_{\pi^\flat}(\sigma)^{-n} z^n=\chi(\sigma)^n 1\otimes z^n
\]
as expected.

\begin{lemma}
  \label{sec:c_k-semil-galo-3-description-of-essential-image}
  An object $V\in \Rep_C(G_K)$ lies in the essential image of $\alst_K$ if and only if it is nearly Hodge--Tate.
\end{lemma}
Phrased in terms of prismatic crystals this result was conjectured in \cite{min2021hodge} and proved in \cite[Theorem 1.1.5]{gao2022hodge}. 
Note that the $A_1$ appearing in \cite{min2021hodge}, \cite{gao2022hodge} is (probably) equal to $-\Theta_\pi$ by unraveling all identifications.
\begin{proof}
  Assume $(M,\Theta_{\pi,M})\in \mathrm{Vec}(\mathrm{Spf}(\mathcal{O}_K)^{\rm HT})[1/p]$ maps to $V$, i.e.,
  $V=M\otimes_{\mathcal{O}_K} C$ with the action as described in \Cref{sec:c_k-semil-galo-1-alternative-description-of-galois-action}. 
Set $z\in C^\times$ as in \Cref{l:galois-action-on-z}. 
Then
  \[
    \frac{\sigma(z)}{z}=\chi(\sigma)\chi_{\pi^\flat}(\sigma)^{-1}.
  \]
  Passing to a finite field extension $K^\prime$ of $K$ as in \Cref{p:comparison-B_en-B_Sen} yields $z^\prime\in C^\times$ such that
  \[
    \frac{\sigma(z^\prime)}{z^\prime}=\chi(\sigma)\chi_{\pi^\flat}(\sigma)^{-1}.
  \]
  for all $\sigma\in G_K$ and such that $w:=(z^\prime)^{\Theta_{\pi,M}/e}$ converges.
  Then for $\sigma\in G_{K^\prime}$, $m\in M\subseteq V,$
  \[
    \sigma(w\cdot m)=\sigma(w)\chi_{\pi^\flat}(\sigma)^{\Theta_{\pi,M}/e}m=\chi(\sigma)^{\Theta_{\pi,M}/e}wm.
  \]
  This implies that $e\Theta=\Theta_{\pi,M}$ by the unique characterisation of the classical Sen operator $\Theta$ for $V$, cf.\ \cite[Theorem 4]{sen1980continuous}. 
In particular, $\Theta$ has eigenvalues in $\Z+E^\prime_\pi(\pi)\mathfrak{m}_{\overline{K}}$ because $\Theta_{\pi,M}$ has eigenvalues in $e\cdot\Z+\mathfrak{m}_{\overline{K}}$ by \Cref{sec:complexes-Hodge--Tate-1-complexes-on-ht-locus-for-r}.
  Conversely, assume that $V$ is nearly Hodge--Tate. 
Then for any $\sigma\in G_K$ the sum
  \[
    \psi_\sigma:=\chi_{\pi^\flat}(\sigma)^{-\Theta}\colon V\to V
  \]
  converges (see the proof of \Cref{sec:complexes-Hodge--Tate-explicit-action}). 
Now define a new $G_K$-action on $V$ by setting
  \[
    \sigma \ast v:=\psi_\sigma(\sigma(v)).
  \]
  Again by classical Sen theory, we can find $v_1,\ldots, v_n\in V$ such that for a suitable open subgroup $U\subseteq G_K$, any $\sigma\in U$ acts via
  \[
    \sigma(v_i)=\chi(\sigma)^\Theta v_i=\exp(\Theta\cdot \log(\chi(\sigma)))v_i.
  \]
  Using multiplication by $w^{\Theta}$ on the $K$-span of $v_1,\ldots, v_n$ we see that $V^{U,\ast}$, the space of invariants of $U$ for the action via $\ast$, has dimension $n$. 
In particular, $V$ for the action via $\ast$ has trivial Sen operator. 
But this implies that it is generated by its $G_K$-invariants $M$ (more precisely, $M\otimes_K C\to V$ is an isomorphism for $M:=V^{G_K,\ast}$), cf.\ \cite[Theorem 6]{sen1980continuous}. 
By construction, the given action on $V$ equips $M$ with an action via $\chi_{\pi^\flat}(\sigma)^{\Theta}$, i.e., $V$ is associated with $(M,e\Theta)\in \calPerf(BG_\pi)$. 
This finishes the proof. 
\end{proof}

\begin{remark}
\label{relation-with-sen}
Alternatively, one could use \Cref{p:comparison-B_en-B_Sen-bis} to relate $\Theta_{\pi,M}$ and $\Theta$.
\end{remark}

\begin{remark}
\label{sec:comp-diff-unif}
The construction of $\alst_K$ is independent of any choice, but $\beta$ depends on the choice of a uniformizer $\pi$.
Conversely, we can use these statements to clarify the way this description of perfect complexes depends on these choices. 
Namely, let $\tilde{\pi}\in \mathcal{O}_K$ be a second uniformizer. 
Then there exists a unique isomorphism
\[
  \delta_{\pi,\tilde{\pi}}\colon G_{\tilde{\pi}}\to G_{\pi},
\]
which is compatible with the two respective projections to $\mathbb{G}_m^\sharp$. 
Explicitly,
\[
  \delta_{\pi,\tilde{\pi}}(t,a)=(t, \frac{\tilde{e}}{e} a),
\]
where $\tilde{e}:=E^\prime_{\tilde{\pi}}(\tilde{\pi})$. 
The functor
\[
  \delta^\ast_{\pi, \tilde{\pi}}\colon \mathcal{D}(BG_\pi)\to \mathcal{D}(BG_{\tilde{\pi}})
\]
sends a complex $(M,\Theta_{\pi,M})$ to the pair $(M,\frac{e}{\tilde{e}} \Theta_{\pi,M})$.
But there is also a different equivalence, more relevant to us. 
Namely, let
\[
  \rho_\pi\colon \Spf(\O_K)\to \mathrm{Spf}(\mathcal{O}_K)^{\rm HT},\ \rho_{\tilde{\pi}}\colon \Spf(\O_K)\to \mathrm{Spf}(\mathcal{O}_K)^{\rm HT}
\]
be the two morphisms associated with $\pi, \tilde{\pi}$. 
Then we obtain an equivalence
\[
 \beta_{\pi,\tilde{\pi}}\colon BG_\pi\cong \mathrm{Spf}(\mathcal{O}_K)^{\rm HT}\cong BG_{\tilde{\pi}},
\]
which does \textit{not} agree with $\delta_{\pi,\tilde{\pi}}$. 
For example, $\beta_{\pi,\tilde{\pi}}$ does not map the trivial $G_\pi$-torsor to the trivial $G_{\tilde{\pi}}$-torsor.

Fix compatible systems $\pi^\flat$, $\tilde{\pi}^\flat$, which yield identifications between the three morphisms
\[
  \tilde{f}:\Spf(\O_C)\to \mathrm{Spf}(\mathcal{O}_K)^{\rm HT},\]
  \[
   \Spf(\O_C)\to \Spf(\O_K) \xrightarrow{\rho_\pi} \mathrm{Spf}(\mathcal{O}_K)^{\rm HT},\]
   \[ \Spf(\O_C)\to \Spf(\O_K) \xrightarrow{\rho_{\tilde{\pi}}} \mathrm{Spf}(\mathcal{O}_K)^{\rm HT}.
\]

Let $\mathcal{E}\in \Vec(\mathrm{Spf}(\mathcal{O}_K)^{\rm HT})[1/p]$ and denote by
\[
(M_\pi, \Theta_\pi)\in \Vec(BG_\pi)[1/p],  ~~ (M_{\tilde{\pi}}, \Theta_{\tilde{\pi}})\in \Vec(BG_{\tilde{\pi}})[1/p]
\]
 the corresponding objects. 
Set $V:=f^\ast \mathcal{E}\in \Rep_C(G_K)$.
Then $\gamma_{\pi^\flat}$ induces an isomorphism $V\cong M_\pi\otimes_K C$ with action via $\chi_{\pi^\flat}$, cf.\ \Cref{sec:c_k-semil-galo-1-alternative-description-of-galois-action}. 
Now \Cref{sec:c_k-semil-galo-1-pullback-fully-faithful-up-to-quasi-isogeny} and the proof of \Cref{sec:c_k-semil-galo-3-description-of-essential-image} imply that $M_\pi$ is the unique $G_K$-stable $K$-subspace of $V$ such that $G_K$ acts on $M_\pi$ via $\chi_{\pi^\flat}(-)^{\Theta}$, where $\Theta\colon V\to V$ is the Sen operator.
Similarly, $M_{\tilde{\pi}}$ identifies with the unique $G_K$-stable $K$-subspace of $V$ on which $G_K$-acts via $\chi_{\tilde{\pi}^\flat}(-)^\Theta$. 
This gives, in principle, a way of constructing $M_{\tilde{\pi}}$ out of $M_\pi$ or vice versa, or in other words describing instances of $\beta^\ast_{\pi, \tilde{\pi}}$:

Set $z:=\theta(\frac{E_\pi([\pi^\flat])}{E_{\tilde{\pi}}([\tilde{\pi}^\flat])})$. 
Then
\[
  \frac{\sigma(z)}{z}=\chi_{\pi^\flat}(\sigma)\chi_{\tilde{\pi}^\flat}(\sigma)^{-1}
\]
for $\sigma\in G_K$. 
If $z^{\Theta}$ is well-defined, then
\[
  M_{\pi}=z^\Theta\cdot M_{\tilde{\pi}}
\]
as subspaces of $V$. 
Thus, in this case the relation between $M_\pi$ and $M_{\tilde{\pi}}$ is more direct.
\end{remark}

We finally arrive out our desired description of the whole category $\calPerf(\Spa(K)_v)$.

\begin{theorem}
  \label{full-description-of-perf-spd-k}
  For any finite Galois extension $L/K$ the functor
  \[
   \alst_L\colon \calPerf(\mathrm{Spf}(\mathcal{O}_L)^{\rm HT})[1/p]\to \calPerf(\Spa (L)_v)
  \]
  is fully faithful and induces a fully faithful functor
  \[
   \alst_{L/K} \colon \calPerf([\mathrm{Spf}(\mathcal{O}_L)^{\rm HT}/{\Gal(L/K)}])[1/p] \to \calPerf(\Spa(K)_v)
  \]
  on $\Gal(L/K)$-equivariant objects. 
Each $\mathcal{E}\in \calPerf(\Spa(K)_v)$ lies in the essential image of some $\alst_{L/K}$. 
Consequently, we get an equivalence
  \[
   2\text{-}\varinjlim\limits_{L/K} \calPerf([\mathrm{Spf}(\mathcal{O}_L)^{\rm HT}/\mathrm{Gal}(L/K)])[1/p] \cong \calPerf(\mathrm{Spa}(K)_v),
   \]
   where $L$ runs over finite Galois extensions  of $K$ contained in $\overline{K}$. 
\end{theorem}
\begin{proof}
  Full faithfulness of $\alst_L$ is just \Cref{sec:c_k-semil-galo-1-pullback-fully-faithful-up-to-quasi-isogeny} applied to $K=L$. 
By naturality of the Hodge--Tate stacks the functor $\alst_L$ passes to $\Gal(L/K)$-equivariant objects, and the resulting functor $\alst_{L/K}$ is again fully faithful by finite \'etale descent of perfect complexes and full faithfulness of $\alst_L$. 
If $V\in \Rep_C(G_K)$, then its restriction $W:=V_{|G_L}$ to some $G_L$ for $L/K$ a large enough finite Galois extension will become nearly Hodge--Tate (with respect to $L$), and hence lies in the essential image of $\alst_L$. 
As $W\in \calPerf(\Spa (L)_v)$ is $\Gal(L/K)$-equivariant and $\alst_L$ fully faithful, the final claim follows (note that each $\mathcal{E}\in \calPerf(\Spa (K)_v)$ can be represented as a complex of continuous $G_K$-representations).
\end{proof}

\begin{remark}
\label{description-pro-system-of-stacks}
More geometrically, one could consider the pro-system
\[
  Y:= "\varprojlim" [\mathrm{Spf}(\mathcal{O}_L)^{\rm HT}/\Gal(L/K)]
\]
over finite Galois extensions $L/K$ contained in $\overline{K}$ and consider the morphism
\[
 \tilde{f}_{\mathrm{pro}} \colon [\Spf(\mathcal{O}_C)/G_K]\to Y,
\]
obtained by the natural lifts $\Spf(\mathcal{O}_C)\to \mathrm{Spf}(\mathcal{O}_L)^{\rm HT}$.
Probably, it makes sense to pass to a ``generic fiber $\mathrm{Spa}(K)^{\rm HT}$'' of $Y$ and define a reasonable notion of perfect complexes on $\mathrm{Spa}(K)^{\rm HT}$ (for the analytic topology) in such a way that
\[
  \calPerf(\mathrm{Spa}(K)^{\rm HT})\cong 2\text{-}\varinjlim\limits_{L/K} \calPerf([\mathrm{Spf}(\mathcal{O}_L)^{\rm HT}/\mathrm{Gal}(L/K)])[1/p],
\]
cf.\ \Cref{sec:c_k-semil-galo-2-generic-fiber-of-Hodge--Tate-stack}.
Then \Cref{full-description-of-perf-spd-k} could more cleanly be stated as an equivalence.
\[
  \calPerf(\mathrm{Spa}(K)^{\rm HT})\cong \calPerf(\Spa (K)_v).
\]
This would yield a description of v-bundles on $\Spa(K)$ in terms of the \'etale/analytic topology on a certain analytic stack over $\Spa(K)$. We plan to come back to this question in future work.
\end{remark}

\section{The $\mathrm{v}$-Picard group of $p$-adic fields}
\label{sec:v-picard-group-of-local-fields}

As pointed out in the introduction, \Cref{sec:c_k-semil-galo-1-pullback-fully-faithful-up-to-quasi-isogeny} (together with \Cref{sec:c_k-semil-galo-3-description-of-essential-image}) can be regarded as an analogue of the local Simpson correspondence, in that it relates ``small'' semilinear $C$-representations of $G_K$ to ``small'' Sen modules.

\Cref{full-description-of-perf-spd-k} provides (formally) a description of all $G_K$-representations. One might still ask whether there exists a simpler, more classical description. As already mentioned in the introduction, one usually cannot hope for an \textit{equivalence} of categories between v-bundles and Sen modules, but is it still possible to give a more precise relation between the whole two categories?

As a first partial answer in this direction, we note that it is indeed possible in general to extend the local correspondence between small objects, by the following immediate consequence of \Cref{sec:c_k-semil-galo-3-description-of-essential-image} and \Cref{sec:c_k-semil-galo-1-pullback-fully-faithful-up-to-quasi-isogeny}:
\begin{corollary}\label{c:nearly-HT-into-Sen-modules}
	There is a natural fully faithful functor\footnote{We note that the notion of ``nearly Hodge--Tate'' for a representation is independent of any choice and hence passes to an intrinsic notion for v-bundles on $\Spa(K)$.}
	\[\Bigg\{\begin{array}{@{}c@{}l}\text{nearly Hodge--Tate}\\\text{v-vector bundles on $\Spa(K)$}\end{array}\Bigg\}
	\hookrightarrow 
	\Bigg\{\begin{array}{@{}c@{}l}\text{Sen modules over K}\end{array}\Bigg\}.\]
\end{corollary}

One might hope optimistically that there is an equivalence between all v-bundles and all Higgs bundles, like in the geometric situation. 
However, this is not the case.
Indeed, the goal of this section is to prove the following result, which evidences that over a local field, there are always ``fewer v-vector bundles than Higgs bundles''.
\begin{theorem}\label{t:Picv}
	For any complete discretely valued field $K|\Q_p$, there is a short exact sequence
	\[ 0\to \Pic_v(K)\xrightarrow{\HT\log} K\to\Br(K)[p^\infty]\to 0\]
	that is functorial in $K$.
	If $K|\Q_p$ is finite, then $\Br(K)[p^\infty]\cong \Q_p/\Z_p$ by local class field theory and the last map is given by
	\[\tfrac{u}{p}\mathrm{Tr}_{K|\Q_p}: K\to \Q_p/\Z_p\]
        for some $u\in \Z_p^\times$, which is independent of $K$.
\end{theorem}
The map $\HT\log$ will be defined in the proof. 
It has already been described by Sen \cite[Theorem~7']{sen1980continuous}, as well as by Serre \cite[\S III, A2, Proposition~2]{serre1997abelian}, in terms of the continuous cohomology $\Pic_v(K)=H^1(G_K,C^\times)$. 
Sen also shows that if the residue field is algebraically closed, then $\HT\log$ is an isomorphism \cite[Theorem~9']{sen1980continuous}. 
From the perspective of \Cref{t:Picv}, this is because $\Br(K)=1$ in this case.
Beyond this case, we are not aware of any previous description of the cokernel of $\HT\log$.
\begin{proof}
	Consider the morphism of sites $\nu:\Spa(K)_v\to \Spa(K)_{\et}$. 
By Tate's cohomological result, we have $R^1\nu_{\ast}\O=\O$. 
We instead consider the sheaf $\Gm$ on $\Spa(K)_v$ defined by $T\mapsto \O(T)^\times$. 
Following \cite[\S III, A2]{serre1997abelian} or alternatively arguing as in \cite[\S2]{vlinebundles}, we see:
	
	\begin{lemma}
		The exponential induces a natural isomorphism $R^1\nu_{\ast}\Gm=\O$.
	\end{lemma}
	\begin{proof}
	It suffices to construct natural isomorphisms
	\[R^1\nu_{\ast}(1+\mathfrak m \O^+)\xrightarrow{\sim}  R^1\nu_{\ast}\Gm\]
	\[\log:R^1\nu_{\ast}(1+\mathfrak m \O^+)\xrightarrow{\sim}  R^1\nu_{\ast}\O.\]
	Namely let $K_\infty|K$ be any totally ramified perfectoid Galois cover with group $\Delta\cong \Z_p$. 
Then the first isomorphism is obtained by applying $H^1_{}(\Delta,-)$ to the short exact sequences
	\[ 0\to \O_{K_\infty}^\times\to K_\infty^{\times} \to \Z[\tfrac{1}{p}]\to 0\]
	\[ 0\to 1+\mathfrak m_{K_\infty}\to \O_{K_\infty}^{\times} \to k^\times\to 0\]
	and using that $H^1_{}(\Delta,k^\times)=\Hom_{\cts}(\Delta,k^\times)=1$ and $H^1_{}(\Delta,\Z[\tfrac{1}{p}])=1$, where $k$ is the residue field. 
The second isomorphism then follows from the logarithm sequence over $C=\widehat{\overline{K}}$.
	\end{proof}

	Since $\Pic_{\et}(K)=1$, it follows that the Leray sequence of $\Gm$ for $\nu$ is of the form
	\[ 1\to \Pic_v(K)\xrightarrow{\HT\log} K\xrightarrow{\partial_K}  H^2_{\et}(K,\mathbb G_m)\to H^2_{v}(K,\mathbb G_m),\]
	functorially in $K$. 
Let $Q:=\operatorname{coker} \HT\log=\ker \partial_K$.
	It follows from the construction that $\HT\log$ admits a  canonical splitting over $p\O_K\subseteq K$ defined by the $p$-adic exponential function, hence $Q$ is a $p$-power torsion group. 
This shows that $Q\subseteq \Br(K)[p^\infty]$.
	
	To see the other containment, it now suffices to prove:
	\begin{lemma}
	$H^2_{v}(K,\mathbb G_m)$ is $p$-torsionfree.
	\end{lemma}
	\begin{proof}
		The Cartan--Leray sequence of the $\Delta$-torsor $\Spa(K_\infty)\to \Spa(K)$ is of the form
		\[ E^{nm}_2=H^n_{}(\Delta,H^m_v(K_\infty,\mathbb G_m))\Rightarrow H^{n+m}_v(K,\mathbb G_m).\]
		Since $H^1_v(K_\infty,\Gm)=H^1_{\et}(K_\infty,\Gm)=1$, the column $E^{\bullet 1}_2$ vanishes. 
Using the above short exact sequences, we moreover see that for any $n\geq 2$,
		\[ H^n_{}(\Delta,K_\infty^\times)=H^n_{}(\Delta,1+\mathfrak m_{K_\infty})=1,\]
		because we can write $1+\mathfrak m_{K_\infty}=\varinjlim_{n\to \infty}\{x\in K_\infty\mid |x-1|\leq |p|^\frac{1}{p^n}\}$ as a colimit of $p$-complete $\Z_p$-modules, for which $H^n(\Delta,-)$ vanishes for $n\geq 2$. 
Thus the first column $E^{\bullet 0}_2$ vanishes in degree $n\geq 2$. 
It follows that the natural map
		\[H^2_{v}(K,\mathbb G_m)\hookrightarrow H^2_{v}(K_\infty,\mathbb G_m)\]
		is injective. 
But $H^2_{v}(K_\infty,\mathbb G_m)$ is $p$-torsionfree since $H^2_{v}(K_\infty,\mu_p)=H^2_{\et}(K_\infty,\mu_p)$ vanishes for the perfectoid field $K_\infty$ by \cite[Theorem~4.10]{vcesnavivcius2019purity}.
	\end{proof}
	
	We thus get the first exact sequence, and it remains to identify the transition morphism in the case of local fields:
	recall from local class field theory that there is a natural isomorphism
	\[ \mathrm{inv}_K:\Br(K)\to \Q/\Z\]
	such that for any extension $L|K$, we have  $\mathrm{inv}_L(x)=[L:K]\cdot \mathrm{inv}_K(x)$ for any $x\in \Br(K)$. 
Via this identification, we obtain a morphism, functorial in $K$,
	\[ \partial_K:K\to\Q/\Z.\]
	It follows formally that this morphism must factor through the trace $\mathrm{Tr}_{K|\Q_p}:K\to \Q_p$: namely, assume without loss of generality that $K|\Q_p$ is Galois, then the functoriality in $K$ means that $\partial_K$ is $G:=\Gal(K|\Q_p)$-equivariant. 
Let $n:=|G|=[K:\Q_p]$, then for any $x\in K$, we thus have
	\[ \partial_K(x)=\partial_K(\sum_{\sigma\in G}\frac{1}{n}x)=\sum_{\sigma\in G}\partial_K(\frac{x}{n})=\sum_{\sigma\in G}\partial_K(\frac{\sigma(x)}{n})=\partial_K(\sum_{\sigma\in G}\frac{\sigma(x)}{n})=\partial_K(\frac{1}{n}\mathrm{Tr}_{K|\Q_p}(x)).\]
	Since $\frac{1}{n}\mathrm{Tr}_{K|\Q_p}(x)\in \Q_p$, this now equals
	\[ =[K:\Q_p]\partial_{\Q_p}(\frac{1}{n}\mathrm{Tr}_{K|\Q_p}(x))=\partial_{\Q_p}(\mathrm{Tr}_{K|\Q_p}(x)).\]
	We can thus reduce to computing the homomorphism
	\[ \partial_{\Q_p}:\Q_p\to \Q_p/\Z_p\]
	which is necessarily given by multiplication by some uniquely determined $a\in \Q_p$, followed by the quotient map. We claim that $a\in \frac{1}{p}\Z_p^\times$, or equivalently, $\mathrm{ker}(\Q_p\overset{\cdot a}{\to}\Q_p\to \Q_p/\Z_p)=p\Z_p$. This implies the remaining assertion.
It therefore suffices in a second step to reduce to $K=\Q_p(\mu_p)$, where some computations become easier.
	
	In order to determine the element $a$, we consider the Galois cohomology of the short exact sequence
	\[ 0\to \mu_{p^\infty}\to 1+\mathfrak m_C\to C\to 0.\]
	Let $G_K:=\Gal_{\mathrm{}}(C|K)$, then by tracing through the identification $R^1\nu_{\ast}\O^\times=\O$, we easily verify that this induces an exact sequence
	\[
	\begin{tikzcd}
		{H^1(G_K,1+\mathfrak m_C)} \arrow[d, two heads] \arrow[r,"\log"] & {H^1(G_K,C)} \arrow[d,"\sim"] \arrow[r,"\partial'"] & {H^2(G_K,\mu_{p^\infty})} \arrow[d,hook] \\
		\Pic_v(K) \arrow[r]                              & K \arrow[r]                      & H^2(G_K,C^\times),                             
	\end{tikzcd}
	\]
	in which the first vertical map is surjective, the second is an isomorphism, and the third is injective with image the $p$-primary part of the Brauer group.
	
	It thus suffices to identify the morphism $\partial'$. 
This is a boundary morphism for Galois cohomology, which we can make explicit: Let $\Q_p^{cyc}| \Q_p(\mu_p)=K$  be the completed cyclotomic extension with Galois group $\Gamma:=\Z_p$. 
Then any element of $H^1(G_K,C)$ is represented by a continuous homomorphism $\Gamma\to K$ given by $\gamma\mapsto \gamma \cdot b$ for some element $b\in K$. 
We note that since $\log$ is split by $\exp$ over $p\O_K$, it is clear that $\partial'(b)=0$ for $b\in p\O_K$. 
	On the other hand, let $b\in K$ be any element and let $y\in 1+\mathfrak m_C$ be a preimage of $b$ under $\log$. 
Then $\partial'$ sends $y$ to the $2$-cocycle $\Gamma^2\to \mu_{p^\infty}$
	\begin{equation}\label{eq:expl-descr-of-2-cocycle-obstr-for-Pic}
	(g,h)\mapsto y^g \cdot g^{\ast}(y^h)\cdot (y^{gh})^{-1}=(g^\ast y/y)^h.
	\end{equation}
	We see from this explicit description that this $2$-cocycle can be captured as follows: consider the extension
	\[
	\begin{tikzcd}
		{\Q_p^\cycl(y^{1/p^\infty})} \arrow[d, "\Delta", no head] \arrow[dd, "H", no head, bend left=49] \\
		\Q_p^\cycl \arrow[d, "\Gamma", no head]                                                                         \\
		K                                                                                                          
	\end{tikzcd}\]
with Galois groups as indicated. 
Due to the functoriality properties of $\mathrm{inv}$, we know that any $p^n$ torsion class in $\Br(K)$ is killed by an extension of $K$ of degree $p^n$. 
It follows that the kernel of the map
\[\Br(K)[p^\infty]=H^2_{v}(\Spa(K),\mu_{p^\infty})\to H^2_{v}(\Spa(\Q_p^\cycl(y^{1/p^\infty})),\mu_{p^\infty})\]
contains all of $\Br(K)[p^\infty]$. 
By a Hochschild--Serre sequence, and due to vanishing of $H^i_{}(\Gamma,-)$ and $H^i_{}(\Delta,-)$ for any $i\geq 2$, we get a natural identification
\begin{equation}\label{eq:cycl-descr-of-Br}
\Br(K)[p^\infty] \xrightarrow{\sim} H^1_{}(\Gamma,H^1_{}(\Delta,\mu_{p^\infty}))=\Q_p/\Z_p,
\end{equation}
where the last map comes from the identification $\Delta=\Z_p(1)$.

It is now clear from construction that the  explicit cocycle (\eqref{eq:expl-descr-of-2-cocycle-obstr-for-Pic}) defines an element in the middle term. 
To see when this vanishes in $\Br(K)$, it thus suffices to see when the $1$-cocycle
\[ g\mapsto g^{\ast}y/y\]
is non-trivial. 
Via the identification from the Kummer sequence
\[ H^1(\Q_p^\cycl,\mu_{p^\infty})=\Q_p^{\cycl\times}\otimes \Q_p/\Z_p,\]
we see that this happens if and only if $y$ is already contained in $\Q_p^\cycl$. 
We claim that this is never the case for $b\not\in p\Z_p$, for which it suffices to prove that no element $b\in K$ with $|b|=|p|$ is such that $\exp(b)$ has a $p$-th  root in $\Q_p^\cycl$. 
For this we use that the natural map
\[ \Q_p^\times/ \Q_p^{\times p}\hookrightarrow\Q_p^{\cycl\times}/ \Q_p^{\cycl\times p}\]
is injective: Its kernel can be described via the Kummer sequence and the Cartan--Leray sequence as the cohomology group $H^1_{}(\Z_p^\times,\mu_{p})$ for the action of $\Z_p^\times$ on $\mu_{p}$ via $c\cdot\zeta_p:=\zeta_p^c$, $c\in \Z_p^\times$, which vanishes. 
Thus $\exp(b)$ has a $p$-th root in $\Q_p^{\cycl\times}$ if and only if it has as in $\Q_p$. 
Such a root cannot exist due to the assumption that $|b|=|p|$, which would force $0<|1-\exp(b^{1/p})|<1$, a contradiction to $b^{1/p}\in \Q_p$.
\end{proof}

\begin{remark}
	Does the natural map \eqref{eq:cycl-descr-of-Br} agree with (the $p$-primary map of) the canonical map $\mathrm{inv}:\Br(K)\to \Q/\Z$ from local class field theory? This seems plausible, but a bit cumbersome to prove since this morphism is constructed using the maximal unramified extension rather than the highly ramified one we consider above. 
In any case, we do not need it for the description of the image of the map $\HT\log:\Pic_v(K)\to K$.
\end{remark}
Since any morphism between v-line bundles on $K$ is either $0$ or an isomorphism, and the same is true for Sen modules,  we deduce from \Cref{t:Picv} that in the case of rank one,  \Cref{c:nearly-HT-into-Sen-modules} yields:
\begin{corollary}\label{c:cor-to-PicvK}
	There is a fully faithful functor, which is however not essentially surjective:
		\[\Big\{\begin{array}{@{}c@{}l}\text{v-line bundles on $\Spa(K)$}\end{array}\Big\}
	\hookrightarrow 
	\Big\{\begin{array}{@{}c@{}l}\text{Sen modules over K}\\\text{of rank one}\end{array}\Big\}.\]
\end{corollary}
 \Cref{c:cor-to-PicvK} gives another reason why we cannot expect to have an equivalence of categories between all v-vector bundles and all Sen modules over $\Spa(K)$ that is compatible with both the functor from \Cref{c:nearly-HT-into-Sen-modules} as well as passage to finite extensions. 
In particular, this means that the immediate analogue of the global Simpson correspondence fails in the arithmetic setting of smooth proper rigid spaces over $K$. Moreover, it shows that the relation between v-vector bundles and Sen modules is very subtle and sees non-trivial arithmetic information about $K$ like the $p$-primary part of its Brauer group already in the case of rank one.

\end{document}